\documentclass{article}
\usepackage{amsmath}
\usepackage{amssymb}
\usepackage{amsfonts}

\newtheorem{theorem}{Theorem}[section]

\newtheorem{corollary}{Corollary}
\newtheorem{criterion}{Criterion}[section]
\newtheorem{definition}[theorem]{Definition}

\newtheorem{proposition}{Proposition}
[section]

\newenvironment{proof}[1][Proof]{\noindent\textbf{#1.} }{\ \rule{0.5em}{0.5em}}

\input{tcilatex}
\begin{document}

\title{{\Large Strongly Stable Automorphisms of the Categories of the
Finitely Generated Free Algebras of the some Varieties of Linear Algebras.}}
\author{{\Large A.Tsurkov} \\
Institute of Mathematics and Statistics.\\
University S\~{a}o Paulo. \\
Rua do Mat\~{a}o, 1010 \\
Cidade Universit\'{a}ria \\
S\~{a}o Paulo - SP - Brazil - CEP 05508-090 \\
arkady.tsurkov@gmail.com}
\maketitle

\begin{abstract}
In this paper we consider some classical varieties of linear algebras over
the field $k$ such that $char\left( k\right) =0$. If we denote by $\Theta $
one of these varieties, then $\Theta ^{0}$ is a category of the finite
generated free algebras of the variety $\Theta $. In this paper we calculate
for the considered varieties the quotient group $\mathfrak{A/Y}$, where $%
\mathfrak{A}$ is a group of the all automorphisms of the category $\Theta
^{0}$ and $\mathfrak{Y}$ is a subgroup of the all inner automorphisms of
this category. The quotient group $\mathfrak{A/Y}$ measures difference
between the geometric equivalence and automorphic equivalence of algebras
from the variety $\Theta $. The results of this paper and of the \cite%
{Tsurkov} are summarized in the table in the Section \ref{table}.
\end{abstract}

\section{Introduction.\label{intro}}

\setcounter{equation}{0}

This is a paper from universal algebraic geometry. All definitions of the
basic notions of the universal algebraic geometry can be found, for example,
in \cite{PlotkinVarCat}, \cite{PlotkinNotions} and \cite{PlotkinSame}.

This research is a continuation of the \cite{Tsurkov}.

If we will compare the geometric equivalence and the automorphic equivalence
of the one-sorted universal algebras from the some variety $\Theta $, we
must take a countable set of symbols $X_{0}=\left\{ x_{1},x_{2},\ldots
,x_{n},\ldots \right\} $ and consider all free algebras $F\left( X\right) $
of the variety $\Theta $, generated by finitely subsets $X\subset X_{0}$.
These algebras: $\left\{ F\left( X\right) \mid X\subset X_{0},\left\vert
X\right\vert <\infty \right\} $ - will be objects of the category $\Theta
^{0}$. Morphisms of the category $\Theta ^{0}$ will be homomorphisms of
these algebras.

If our variety $\Theta $ possesses the IBN property: for free algebras $%
F\left( X\right) ,F\left( Y\right) \in \Theta $ we have $F\left( X\right)
\cong F\left( Y\right) $ if and only if $\left\vert X\right\vert =\left\vert
Y\right\vert $ - then we have \cite[Theorem 2]{PlotkinZhitAutCat} the
decomposition%
\begin{equation}
\mathfrak{A=YS}.  \label{decomp}
\end{equation}%
of the group $\mathfrak{A}$ of all automorphisms of the category\textit{\ }$%
\Theta ^{0}$. Hear $\mathfrak{Y}$ is a group of all inner automorphisms of
the category\textit{\ }$\Theta ^{0}$ and $\mathfrak{S}$ is a group of all
strongly stable automorphisms of the category\textit{\ }$\Theta ^{0}$. The
definitions of the notions of inner automorphisms and strongly stable
automorphisms can be found, for example, in \cite{PlotkinZhitAutCat} and 
\cite{Tsurkov}. But we will give these definitions hear.

\begin{definition}
\label{inner}An automorphism $\Upsilon $ of a category $\mathfrak{K}$ is 
\textbf{inner}, if it is isomorphic as a functor to the identity
automorphism of the category $\mathfrak{K}$.
\end{definition}

It means that for every $A\in \mathrm{Ob}\mathfrak{K}$ there exists an
isomorphism $s_{A}^{\Upsilon }:A\rightarrow \Upsilon \left( A\right) $ such
that for every $\psi \in \mathrm{Mor}_{\mathfrak{K}}\left( A,B\right) $ the
diagram%
\begin{equation*}
\begin{array}{ccc}
A & \overrightarrow{s_{A}^{\Upsilon }} & \Upsilon \left( A\right) \\ 
\downarrow \psi &  & \Upsilon \left( \psi \right) \downarrow \\ 
B & \underrightarrow{s_{B}^{\Upsilon }} & \Upsilon \left( B\right)%
\end{array}%
\end{equation*}%
\noindent commutes.

\begin{definition}
\label{str_stab_aut}\textbf{\hspace{-0.08in}. }\textit{An automorphism $\Phi 
$ of the category }$\Theta ^{0}$\textit{\ is called \textbf{strongly stable}
if it satisfies the conditions:}

\begin{enumerate}
\item[StSt1)] $\Phi $\textit{\ preserves all objects of }$\Theta ^{0}$%
\textit{,}

\item[StSt2)] \textit{there exists a system of bijections }$\left\{
s_{F}^{\Phi }:F\rightarrow F\mid F\in \mathrm{Ob}\Theta ^{0}\right\} $%
\textit{\ such that }$\Phi $\textit{\ acts on the morphisms $\psi
:D\rightarrow F$ of }$\Theta ^{0}$\textit{\ by this way: }%
\begin{equation}
\Phi \left( \psi \right) =s_{F}^{\Phi }\psi \left( s_{D}^{\Phi }\right)
^{-1},  \label{biject_action}
\end{equation}

\item[StSt3)] $s_{F}^{\Phi }\mid _{X}=id_{X},$ \textit{\ for every free
algebra} $F=F\left( X\right) $.
\end{enumerate}
\end{definition}

The subgroup $\mathfrak{Y}$ is a normal in $\mathfrak{A}$.

By \cite{PlotkinSame} only strongly stable automorphism $\Phi $ can provide
us automorphic equivalence of algebras which not coincides with geometric
equivalence of algebras. Therefore, in some sense, difference from the
automorphic equivalence to the geometric equivalence is measured by the
quotient group $\mathfrak{A/Y\cong S/S\cap Y}$.

\section{Verbal operations and strongly stable automorphisms.\label%
{operations}}

\setcounter{equation}{0}

In this paper, as in the \cite{Tsurkov} we use the method of verbal
operations for the finding of the strongly stable automorphisms of the
category $\Theta ^{0}$. The explanation of this method there is in \cite%
{PlotkinZhitAutCat} and \cite{TsurkovAutomEquiv}.

We denote the signature of our variety $\Theta $ by $\Omega $, by $m_{\omega
}$ we denote the arity of $\omega $ for every $\omega \in \Omega $. If $%
w=w\left( x_{1},\ldots ,x_{m_{\omega }}\right) \in F\left( x_{1},\ldots
,x_{m_{\omega }}\right) \in \mathrm{Ob}\Theta ^{0}$, then we can define in
every algebra $H\in \Theta $ by using of the this word $w$ the new operation 
$\omega ^{\ast }$: $\omega ^{\ast }\left( h_{1},\ldots ,h_{m_{\omega
}}\right) =w\left( h_{1},\ldots ,h_{m_{\omega }}\right) $ for every $%
h_{1},\ldots ,h_{m_{\omega }}\in H$. This operation we call the \textbf{%
verbal operation} defined on the algebra $H$ by the word $w$. If we have a
system of words $W=\left\{ w_{\omega }\mid \omega \in \Omega \right\} $ such
that $w_{\omega }\in F\left( x_{1},\ldots ,x_{m_{\omega }}\right) $ then we
denote by $H_{W}^{\ast }$ the algebra which coincide with $H$ as a set but
instead the original operations $\left\{ \omega \mid \omega \in \Omega
\right\} $ it has the system of the verbal operations $\left\{ \omega ^{\ast
}\mid \omega \in \Omega \right\} $ defined by words from the system $W$.

We suppose that we have the system of words $W=\left\{ w_{\omega }\mid
\omega \in \Omega \right\} $ satisfies the conditions:

\begin{enumerate}
\item[Op1)] $w_{\omega }\left( x_{1},\ldots ,x_{m_{\omega }}\right) \in
F\left( x_{1},\ldots ,x_{m_{\omega }}\right) \in \mathrm{Ob}\Theta ^{0}$,

\item[Op2)] for every $F=F\left( X\right) \in \mathrm{Ob}\Theta ^{0}$ there
exists an isomorphism $\sigma _{F}:F\rightarrow F_{W}^{\ast }$ such that $%
\sigma _{F}\mid _{X}=id_{X}$.
\end{enumerate}

It is clear isomorphisms $\sigma _{F}$ are defined uniquely by the system of
words $W$.

The set $S=\left\{ \sigma _{F}:F\rightarrow F\mid F\in \mathrm{Ob}\Theta
^{0}\right\} $ is a system of bijections which satisfies the conditions:

\begin{enumerate}
\item[B1)] for every homomorphism $\psi :A\rightarrow B\in \mathrm{Mor}%
\Theta ^{0}$ the mappings $\sigma _{B}\psi \sigma _{A}^{-1}$ and $\sigma
_{B}^{-1}\psi \sigma _{A}$ are homomorphisms;

\item[B2)] $\sigma _{F}\mid _{X}=id_{X}$ for every free algebra $F\in 
\mathrm{Ob}\Theta ^{0}$.
\end{enumerate}

So we can define the strongly stable automorphism\textit{\ }by this system
of bijections. This automorphism preserves all objects of $\Theta ^{0}$ and
acts on morphism of $\Theta ^{0}$ by formula (\ref{biject_action}), where $%
s_{F}^{\Phi }=$ $\sigma _{F}$.

Vice versa if we have a strongly stable automorphism $\Phi $ of the category 
$\Theta ^{0}$ then its system of bijections $S=\left\{ s_{F}^{\Phi
}:F\rightarrow F\mid F\in \mathrm{Ob}\Theta ^{0}\right\} $ defined uniquely.
Really, if $F\in \mathrm{Ob}\Theta ^{0}$ and $f\in F$ then%
\begin{equation}
s_{F}^{\Phi }\left( f\right) =s_{F}^{\Phi }\psi \left( x\right) =\left(
s_{F}^{\Phi }\psi \left( s_{D}^{\Phi }\right) ^{-1}\right) \left( x\right)
=\left( \Phi \left( \psi \right) \right) \left( x\right) ,
\label{autom_bijections}
\end{equation}%
where $D=F\left( x\right) $ - $1$-generated free linear algebra - and $\psi
:D\rightarrow F$ homomorphism such that $\psi \left( x\right) =f$. Obviously
that this system of bijections $S=\left\{ s_{F}^{\Phi }:F\rightarrow F\mid
F\in \mathrm{Ob}\Theta ^{0}\right\} $ fulfills conditions B1) and B2) with $%
\sigma _{F}=s_{F}^{\Phi }$.

If we have a system of bijections $S=\left\{ \sigma _{F}:F\rightarrow F\mid
F\in \mathrm{Ob}\Theta ^{0}\right\} $ which fulfills conditions B1) and B2)
than we can define the system of words $W=\left\{ w_{\omega }\mid \omega \in
\Omega \right\} $ satisfies the conditions Op1) and Op2) by formula%
\begin{equation}
w_{\omega }\left( x_{1},\ldots ,x_{m_{\omega }}\right) =\sigma _{F_{\omega
}}\left( \omega \left( \left( x_{1},\ldots ,x_{m_{\omega }}\right) \right)
\right) \in F_{\omega },  \label{der_veb_opr}
\end{equation}%
where $F_{\omega }=F\left( x_{1},\ldots ,x_{m_{\omega }}\right) $.

By formulas (\ref{autom_bijections}) and (\ref{der_veb_opr}) we can check
that there are

\begin{enumerate}
\item one to one and onto correspondence between strongly stable
automorphisms of the category $\Theta ^{0}$ and systems of bijections
satisfied the conditions B1) and B2)

\item one to one and onto correspondence between systems of bijections
satisfied the conditions B1) and B2) and systems of words satisfied the
conditions Op1) and Op2).
\end{enumerate}

Therefore we can calculate the group $\mathfrak{S}$ if we can find the all
system of words which fulfill conditions Op1) and Op2). For calculation of
the group $\mathfrak{S\cap Y}$ we also have a

\begin{criterion}
\label{inner_stable}The strongly stable automorphism $\Phi $ of the category 
$\Theta ^{0}$ which corresponds to the system of words $W$ is inner if and
only if for every $F\in \mathrm{Ob}\Theta ^{0}$ there exists an isomorphism $%
c_{F}:F\rightarrow F_{W}^{\ast }$ such that $c_{F}\psi =\psi c_{D}$ fulfills
for every $\left( \psi :D\rightarrow F\right) \in \mathrm{Mor}\Theta ^{0}$.
\end{criterion}

\section{Verbal operations in linear algebras.\label%
{operations_in_linear_alg}}

\setcounter{equation}{0}

From now on, the variety $\Theta $ will be some specific variety of the
linear algebras over infinite field $k$, which has the characteristic $0$.
We nether consider the vanished varieties, i., e., variety defined by
identity $x=0$ or variety defined by identity $x_{1}x_{2}=0$.

We consider linear algebras as one-sorted universal algebras, i. e.,
multiplication by scalar we consider as $1$-ary operation for every $\lambda
\in k$: $H\ni h\rightarrow \lambda h\in H$ where $H\in \Theta $. Hence the
signature $\Omega $ of algebras of our variety contains these operations: $0$%
-ary operation $0$; $\left\vert k\right\vert $ $1$-ary operations of
multiplications by scalars; $2$-ary operation $\cdot $ and $2$-ary operation 
$+$. We will finding the system of words $W=\left\{ w_{\omega }\mid \omega
\in \Omega \right\} $ satisfies the conditions Op1) and Op2). We denote the
words corresponding to these operations by $w_{0}$, $w_{\lambda }$ for all $%
\lambda \in k$, $w_{\cdot }$, $w_{+}$. So%
\begin{equation}
W=\left\{ w_{\omega }\mid \omega \in \Omega \right\} =\left\{
w_{0},w_{\lambda }\left( \lambda \in k\right) ,w_{+},w_{\cdot }\right\}
\label{words_list}
\end{equation}%
in our case. From this on we consider only these systems of words.

Some time we denote by $\lambda \ast $ the operation defined by the word $%
w_{\lambda }\left( \lambda \in k\right) $, by $\perp $ the operation defined
by the word $w_{+}$ and by $\times $ the operation defined by the word $%
w_{\cdot }$.

We denote the group of all automorphisms of the field $k$ by $\mathrm{Aut}k$.

We use in our research the familiar fact that every variety of the linear
algebras over infinite field $k$ is multi-homogenous. So, for example, every 
$F\left( X\right) \in \mathrm{Ob}\Theta ^{0}$ can be decompose to the direct
sum of the linear spaces of elements which are homogeneous according the sum
of degrees of generators from the set $X$: $F\left( X\right)
=\bigoplus\limits_{i=1}^{\infty }F_{i}$. We also denote the two sided ideals 
$\bigoplus\limits_{i=j}^{\infty }F_{i}=F^{j}$. $F_{i}F_{j}\subset F_{i+j}$
and $F^{i}F^{j}\subset F_{i+j}$ fulfills for every $1\leq i,j<\infty $. From
now on, the word "ideal" means two sided ideal of linear algebra.

All our varieties $\Theta $ possess the IBN property, because $\left\vert
X\right\vert =\dim F/F^{2}$ fulfills for all free algebras $F=F\left(
X\right) \in \mathrm{Ob}\Theta ^{0}$. So we have the decomposition (\ref%
{decomp}) for group of all automorphisms of the category $\Theta ^{0}$.

\section{Classical varieties of linear algebras.}

\setcounter{equation}{0}

In this Section we consider as the variety $\Theta $ the varieties of the
all commutative algebras, of the all power associative algebras, i., e., the
variety of linear algebras defined by identities%
\begin{equation*}
x\left( x^{2}\right) =\left( x^{2}\right) x,
\end{equation*}%
\begin{equation}
x\left( x\left( x^{2}\right) \right) =x\left( \left( x^{2}\right) x\right)
=\left( x\left( x^{2}\right) \right) x=\left( \left( x^{2}\right) x\right)
x=\left( x^{2}\right) \left( x^{2}\right)  \label{p_ass_id}
\end{equation}%
and so on, of the all alternative algebras, of the all Jordan algebras and
arbitrary subvariety defined by identities with coefficients from $%
\mathbb{Z}
$ of the variety of the all anticommutative algebras.

For the calculating of the group $\mathfrak{S}$ we consider an arbitrary
strongly stable automorphism $\Phi $ of the category $\Theta ^{0}$\ and we
will find for the all possible forms of the system of words $W$ which
corresponds to the automorphism $\Phi $.

For the all considered varieties $F\left( \varnothing \right) =\left\{
0\right\} $, so $w_{0}=0$.

The crucial point is the finding of the words $w_{\lambda }\left( x\right)
\in F\left( x\right) $, where $\lambda \in k$. The system of words $W$ must
fulfills conditions Op1) and Op2). By condition Op2) all axioms of the
variety $\Theta $ must hold in the $F_{W}^{\ast }$ for every $F\in \mathrm{Ob%
}\Theta ^{0}$. For every $\lambda \in k^{\ast }$ must holds 
\begin{equation}
w_{\lambda ^{-1}}\left( w_{\lambda }\left( x\right) \right) =w_{\lambda
}\left( w_{\lambda ^{-1}}\left( x\right) \right) =x.  \label{lambda_invers}
\end{equation}%
So the mapping $F\left( x\right) \ni x\rightarrow w_{\lambda }\left(
x\right) \in F\left( x\right) $ can by extended to the isomorphism.

By \cite{ShirshovComm} the variety of the all commutative algebras is a
Shraier variety and by \cite{Lewin} all automorphisms of the free algebras
of these varieties are tame. So if $\Theta $ is the variety of the all
commutative algebras, then for $\lambda \in k^{\ast }$ we have that 
\begin{equation}
w_{\lambda }\left( x\right) =\varphi \left( \lambda \right) x,
\label{scalar_mult}
\end{equation}
where $\varphi \left( \lambda \right) \in k$. If $\lambda =0$, then must
fulfills $w_{\lambda }\left( x\right) =0$, so in this case we also can write
(\ref{scalar_mult}), where $\varphi \left( \lambda \right) =0$.

If $\Theta $ is the variety of the all power associative algebras, then $%
F\left( x\right) $ is the algebra of the polynomials of degrees no less then 
$1$. Hence from (\ref{lambda_invers}) we can conclude that $\deg w_{\lambda
}\left( x\right) =1$ and (\ref{scalar_mult}) holds. Similar result we have
for the variety of the all alternative algebras and for the variety of the
all Jordan algebras, because these varieties are subvarieties of the variety
of the all power associative algebras (see \cite[Chapter 2, Theorem 2]%
{ZhSlSheShiAlmost} and \cite[Chapter 3, Corollary from Theorem 8]%
{ZhSlSheShiAlmost}), so in these varieties $F\left( x\right) $ is also the
algebra of the polynomials of degrees no less then $1$.

We conclude (\ref{scalar_mult}) for the arbitrary subvariety defined by
identities with coefficients from $%
\mathbb{Z}
$ of the variety of the all anticommutative algebras from the fact that in
this variety $\dim F\left( x\right) =1$. Therefore in all our varieties we
have (\ref{scalar_mult}) for all $\lambda \in k$.

$\lambda \ast \left( \mu \ast x\right) =\left( \lambda \mu \right) \ast x$
must fulfills in $F\left( x\right) $ for every $\lambda ,\mu \in k$. We can
conclude from this axiom as in \cite{Tsurkov} that $\varphi \left( \lambda
\mu \right) =\varphi \left( \lambda \right) \varphi \left( \mu \right) $.
Also by using \cite[Proposition 4.2]{TsurkovAutomEquiv} we can prove that $%
\varphi :k\rightarrow k$ is a surjection.

After this we can conclude from axioms $x_{1}\perp 0=x_{1}$, $0\perp
x_{2}=x_{2}$, $x_{1}\perp x_{2}=x_{2}\perp x_{1}$ and $\lambda \ast \left(
x_{1}\perp x_{2}\right) =\left( \lambda \ast x_{1}\right) \perp \left(
\lambda \ast x_{2}\right) $ as in the \cite{Tsurkov} that in all our
varieties the%
\begin{equation}
w_{+}\left( x_{1},x_{2}\right) =x_{1}+x_{2}  \label{addition}
\end{equation}%
holds. Hear we must use the decomposition of $F\left( x_{1},x_{2}\right) $
to the direct sum of the linear spaces of elements which are homogeneous
according the sum of degrees of generators, which was used in \cite{Tsurkov}.

From axiom $\left( \lambda +\mu \right) \ast x=\lambda \ast x+\mu \ast x$
for every $\lambda ,\mu \in k$ we conclude that $\varphi \left( \lambda +\mu
\right) =\varphi \left( \lambda \right) +\varphi \left( \mu \right) $. So $%
\varphi \in \mathrm{Aut}k$.

Now we must to find the all possible forms of the word $w_{\cdot }\in
F\left( x_{1},x_{2}\right) $. Hear as in \cite{Tsurkov} we use the
decomposition of $F\left( x_{1},x_{2}\right) $ to the direct sum of the
linear spaces of elements which are homogeneous according the degree of $%
x_{1}$, and after this according the degree of $x_{2}$. From axioms $0\times
x_{2}=x_{1}\times 0=0$ and $\lambda \ast \left( x_{1}\times x_{2}\right)
=\left( \lambda \ast x_{1}\right) \times x_{2}=x_{1}\times \left( \lambda
\ast x_{2}\right) $ for every $\lambda \in k$ we can conclude that $w_{\cdot
}\in F_{2}\left( x_{1},x_{2}\right) $, $w_{\cdot }\left( x_{1},0\right)
=w_{\cdot }\left( 0,x_{2}\right) =0$. It means that for the variety of the
all power associative algebras%
\begin{equation}
w_{\cdot }\left( x_{1},x_{2}\right) =\alpha _{1,2}x_{1}x_{2}+\alpha
_{2,1}x_{2}x_{1},  \label{power_assoc_mult}
\end{equation}%
where $\alpha _{1,2},\alpha _{2,1}\in k$. By condition Op2) in this variety
the multiplication in $F_{W}^{\ast }$ can not by commutative or
anticommutative, so $\alpha _{1,2}\neq \pm \alpha _{2,1}$.

For the varieties of the all commutative algebras, of the all Jordan
algebras and for the arbitrary subvariety defined by identities with
coefficients from $%
\mathbb{Z}
$ of the variety of the all anticommutative algebras we have that%
\begin{equation}
w_{\cdot }\left( x_{1},x_{2}\right) =\alpha _{1,2}x_{1}x_{2},
\label{commut_mult}
\end{equation}%
where $\alpha _{1,2}\neq 0$.

For the variety of the all alternative algebras we conclude from axiom $%
\left( x_{1}\times x_{1}\right) \times x_{2}=x_{1}\times \left( x_{1}\times
x_{2}\right) $ that%
\begin{equation}
w_{\cdot }\left( x_{1},x_{2}\right) =\alpha _{1,2}x_{1}x_{2}
\label{altern_mult_1}
\end{equation}%
or%
\begin{equation}
w_{\cdot }\left( x_{1},x_{2}\right) =\alpha _{2,1}x_{2}x_{1},
\label{altern_mult_2}
\end{equation}%
where $\alpha _{1,2},\alpha _{2,1}\neq 0$.

Now we will prove for all our varieties that the systems of words $W$
defined above fulfill condition Op2). First of all we will prove that if $%
H\in \Theta $ then $H_{W}^{\ast }\in \Theta $. It means that we will check
that in the $H_{W}^{\ast }$ the all axioms of the variety $\Theta $ hold.
All these checking can be made by direct calculations. For all our varieties
we must check only these axioms of linear algebra: $\left(
x_{1}+x_{2}\right) \times x_{3}=\left( x_{1}\times x_{3}\right) +\left(
x_{2}\times x_{3}\right) $ and $x_{1}\times \left( x_{2}+x_{3}\right)
=\left( x_{1}\times x_{2}\right) +\left( x_{1}\times x_{3}\right) $, because
other axioms are immediately concluded from the forms of the words of the
system $W$.

For the varieties of the all commutative algebras and of the all Jordan
algebras we must check the axiom $x_{1}\times x_{2}=x_{2}\times x_{1}$.

Also for the variety of the all Jordan algebras we also must check the axiom 
$\left( \left( x_{1}\times x_{1}\right) \times x_{2}\right) \times
x_{1}=\left( x_{1}\times x_{1}\right) \times \left( x_{2}\times x_{1}\right) 
$.

For the variety of the all alternative algebras we must check the axioms $%
\left( x_{1}\times x_{1}\right) \times x_{2}=x_{1}\times \left( x_{1}\times
x_{2}\right) $ and $x_{2}\times \left( x_{1}\times x_{1}\right) =\left(
x_{2}\times x_{1}\right) \times x_{1}$.

For the variety of the all power associative algebras we must check the
axioms (\ref{p_ass_id}) where the original multiplication changed by
operation $\times $. We consider a power associative algebra $H$ and the
free $1$-generated power associative algebra $F=F\left( x\right) $. We take
a monomial $u\in F$ such that coefficient of $u$ is $1$, and $\deg u=n$. If
in the monomial $u$ we change the original multiplication by operation $%
\times $ than we achieve $u^{\times }\in F_{W}^{\ast }$. It is easy to prove
by induction that $u^{\times }\left( h\right) =\left( \alpha _{1,2}+\alpha
_{2,1}\right) ^{n-1}u\left( h\right) $ holds for every $h\in H$. It finishes
the checking of the necessary axioms.

For the arbitrary subvariety defined by identities with coefficients from $%
\mathbb{Z}
$ of the variety of the all anticommutative algebras we must check the axiom 
$x_{1}\times x_{2}=-$ $1\ast \left( x_{2}\times x_{1}\right) $ and the
specific axioms of this subvariety. Our field $k$ is infinite so all axioms
of this subvariety can by presented in the homogeneous form: $%
\sum\limits_{i=1}^{m}\lambda _{i}u_{i}=0$, where $\lambda _{i}\in 
\mathbb{Z}
$, $u_{i}\in F\left( x_{1},\ldots ,x_{r}\right) $, $F\left( x_{1},\ldots
,x_{r}\right) \in \mathrm{Ob}\Theta ^{0}$, $u_{i}$ are monomials with
coefficients $1$, $\deg u_{i}=n$ for every $i$. $\lambda _{i}\in 
\mathbb{Z}
$ because all operations of the process of the homogenization of the
identities can by made with coefficients from $%
\mathbb{Q}
$. We must check that for every $H\in \Theta $ and every $h_{1},\ldots
,h_{r}\in H$ the $\sum\limits_{i=1}^{m}\lambda _{i}\ast u_{i}^{\times
}\left( h_{1},\ldots ,h_{r}\right) =0$ holds. As above by induction we can
prove that $u_{i}^{\times }\left( h_{1},\ldots ,h_{r}\right) =\alpha
_{1,2}^{n-1}u_{i}\left( h_{1},\ldots ,h_{r}\right) $, so $%
\sum\limits_{i=1}^{m}\lambda _{i}\ast u_{i}^{\times }\left( h_{1},\ldots
,h_{r}\right) =\alpha _{1,2}^{n-1}\sum\limits_{i=1}^{m}\lambda
_{i}u_{i}\left( h_{1},\ldots ,h_{r}\right) =0$.

After all these checking we can conclude that for every $F=F\left( X\right)
\in \mathrm{Ob}\Theta ^{0}$ there exists a homomorphism $\sigma
_{F}:F\rightarrow F_{W}^{\ast }$ such that $\sigma _{F}\mid _{X}=id_{X}$. As
in the \cite{Tsurkov} we can prove that $\sigma _{F}$ is an isomorphism, so
the systems of words defined above fulfill condition Op2). It completes the
calculation of the group $\mathfrak{S}$.

For calculation of the group $\mathfrak{Y}\cap \mathfrak{S}$ we can prove as
in the \cite{Tsurkov} that for the all considered varieties the strongly
stable automorphism $\Phi $ which corresponds to the defined above system of
words $W$ is inner if and only if $\alpha _{2,1}=0$ and $\varphi =id_{k}$.

So, as in the \cite{Tsurkov} we can prove that for the variety of the all
power associative algebras $\mathfrak{A/Y\cong }\left( U\left( k\mathbf{S}_{%
\mathbf{2}}\right) \mathfrak{/}U\left( k\left\{ e\right\} \right) \right) 
\mathfrak{\leftthreetimes }\mathrm{Aut}k$, where $\mathbf{S}_{\mathbf{2}}$
is the symmetric group of the set which has $2$ elements, $U\left( k\mathbf{S%
}_{\mathbf{2}}\right) $ is the group of all invertible elements of the group
algebra $k\mathbf{S}_{\mathbf{2}}$, $U\left( k\left\{ e\right\} \right) $ is
a group of all invertible elements of the subalgebra $k\left\{ e\right\} $,
every $\varphi \in \mathrm{Aut}k$ acts on the algebra $k\mathbf{S}_{\mathbf{2%
}}$ by natural way: $\varphi \left( ae+b\left( 12\right) \right) =\varphi
\left( a\right) e+\varphi \left( b\right) \left( 12\right) $. But there is
an isomorphism of groups%
\begin{equation*}
U\left( k\mathbf{S}_{\mathbf{2}}\right) \ni ae+b\left( 12\right) \rightarrow
\left( a+b,a-b\right) \in k^{\ast }\times k^{\ast }
\end{equation*}%
so there is isomorphism%
\begin{equation*}
U\left( kS_{2}\right) /U\left( k\left\{ e\right\} \right) =U\left(
kS_{2}\right) /k^{\ast }e\ni \left( ae+b\left( 12\right) \right) k^{\ast
}e\rightarrow \frac{a+b}{a-b}\in k^{\ast }.
\end{equation*}%
Hence we prove the

\begin{theorem}
\label{group_for_power_associative copy(1)}For variety of the all power
associative algebras 
\begin{equation*}
\mathfrak{A/Y\cong }k^{\ast }\mathfrak{\leftthreetimes }\mathrm{Aut}k
\end{equation*}%
holds.
\end{theorem}

By similar way we prove

\begin{theorem}
\label{group_for_alternative}For the variety of the all alternative algebras%
\begin{equation*}
\mathfrak{A/Y\cong }\mathbf{S}_{\mathbf{2}}\times \mathrm{Aut}k
\end{equation*}%
holds.
\end{theorem}

And for other considered varieties we achieve

\begin{theorem}
\label{group_for_commutative}For the variety of the all commutative algebras%
\begin{equation*}
\mathfrak{A/Y\cong }\mathrm{Aut}k
\end{equation*}%
holds.
\end{theorem}

\begin{theorem}
\label{groupe_for_Jordan}For the variety of the all Jordan algebras 
\begin{equation*}
\mathfrak{A/Y\cong }\mathrm{Aut}k
\end{equation*}%
holds.
\end{theorem}

\begin{theorem}
\label{group_for_anticommutative}For the arbitrary subvariety defined by
identities with coefficients from $%
\mathbb{Z}
$ of the variety of the all anticommutative algebras 
\begin{equation*}
\mathfrak{A/Y\cong }\mathrm{Aut}k
\end{equation*}%
holds.
\end{theorem}

\section{Varieties of nilpotent algebras.\label{Nilp}}

\setcounter{equation}{0}

We denote by $\mathfrak{P}_{n}$ the set of all arrangements of the brackets
on monomial which has $n$ factors. For every $\mathfrak{p\in P}_{n}$ we
denote by $\mathfrak{p}\left( x_{i_{1}},\ldots ,x_{i_{n}}\right) $ the
monomial from absolutely free linear algebra $\widetilde{F}\left( X\right) $%
, where $\left\{ x_{i_{1}},\ldots ,x_{i_{n}}\right\} \subset X$. This
monomial we obtain by putting brackets according the arrangement $\mathfrak{p%
}$ on associative word $x_{i_{1}}\ldots x_{i_{n}}$. In this notation the
variety of all nilpotent linear algebras of degree not more than $n$ is a
variety of linear algebras defined by identities 
\begin{equation}
\mathfrak{p}\left( x_{1},\ldots ,x_{n}\right) =0,  \label{nilp_identities}
\end{equation}%
where $\mathfrak{p\in P}_{n}$. These varieties we denote in this section by $%
\Theta _{n}$, where $n\geq 3$. The categories of the finitely generated free
algebras of these varieties we denote by $\Theta _{n}^{0}$. By $\mathfrak{A}%
_{n}$ we denote the group of the all automorphisms of the category $\Theta
_{n}^{0}$, by $\mathfrak{Y}_{n}$ the group of the all inner automorphisms of
this category and by $\mathfrak{S}_{n}$ the group of the all strongly stable
automorphisms of the category\textit{\ }$\Theta _{n}^{0}$

\subsection{Connection between groups $\mathfrak{A}_{n}\mathfrak{/Y}_{n}$ of
the categories $\Theta _{n}^{0}$.}

In this subsection we develop the method of \cite[Section 2]%
{TsurkovNilpGroup}.

\begin{proposition}
\label{F*i=Fi}If $F\in \mathrm{Ob}\Theta _{n}^{0}$ and the system of words $%
W $ fulfills conditions Op1) and Op2) in $\Theta _{n}^{0}$ then $%
F^{i}=\left( F_{W}^{\ast }\right) ^{i}$ holds for $1\leq i\leq n-1$.
\end{proposition}

\begin{proof}
By condition Op1) we have that $w_{0}=0$ and 
\begin{equation*}
w_{\lambda }\left( x\right) =\varphi \left( \lambda \right) x+l_{2}\left(
x\right) \in F\left( x\right) ,
\end{equation*}%
where $l_{2}\left( x\right) \in \left( F\left( x\right) \right) ^{2}$, $%
\varphi \left( \lambda \right) \in k$. Also 
\begin{equation*}
w_{+}\left( x_{1},x_{2}\right) =\alpha x_{1}+\beta x_{2}+p_{2}\left(
x_{1},x_{2}\right) \in F\left( x_{1},x_{2}\right) ,
\end{equation*}%
where $p_{2}\left( x_{1},x_{2}\right) \in \left( F\left( x_{1},x_{2}\right)
\right) ^{2}$, $\alpha ,\beta \in k$. By Op2) 
\begin{equation*}
w_{+}\left( x_{1},0\right) =\alpha x_{1}+p_{2}\left( x_{1},0\right) =x_{1}
\end{equation*}%
and 
\begin{equation*}
w_{+}\left( 0,x_{2}\right) =\beta x_{2}+p_{2}\left( 0,x_{2}\right) =x_{2}
\end{equation*}%
must fulfill. So $\alpha =\beta =1$. And 
\begin{equation*}
w_{+}\left( x_{1},x_{2}\right) =x_{1}+x_{2}+p_{2}\left( x_{1},x_{2}\right) .
\end{equation*}%
Analogously 
\begin{equation*}
w_{\cdot }\left( x_{1},x_{2}\right) =\gamma x_{1}+\delta x_{2}+q_{2}\left(
x_{1},x_{2}\right) \in F\left( x_{1},x_{2}\right) ,
\end{equation*}%
where $q_{2}\left( x_{1},x_{2}\right) \in \left( F\left( x_{1},x_{2}\right)
\right) ^{2}$, $\gamma ,\delta \in k$.%
\begin{equation*}
w_{\cdot }\left( x_{1},0\right) =\gamma x_{1}+q_{2}\left( x_{1},0\right) =0
\end{equation*}%
and%
\begin{equation*}
w_{\cdot }\left( 0,x_{2}\right) =\delta x_{2}+q_{2}\left( 0,x_{2}\right) =0
\end{equation*}%
must fulfill. So $\gamma =\delta =0$ and $q_{2}\left( x_{1},0\right)
=q_{2}\left( 0,x_{2}\right) =0$. $q_{2}\left( x_{1},0\right) $ is a sum of
monomials which contain only $x_{1}$ and $q_{2}\left( x_{2},0\right) $ is a
sum of monomials which contain only $x_{2}$. Therefore%
\begin{equation*}
w_{\cdot }\left( x_{1},x_{2}\right) =\alpha _{1,2}x_{1}x_{2}+\alpha
_{2,1}x_{2}x_{1}+q_{3}\left( x_{1},x_{2}\right) 
\end{equation*}%
where $q_{3}\left( x_{1},x_{2}\right) \in \left( F\left( x_{1},x_{2}\right)
\right) ^{3}$, $\alpha _{1,2},\alpha _{2,1}\in k$.

Now we can prove that 
\begin{equation}
\left( F_{W}^{\ast }\right) ^{i}\subset F^{i}  \label{F*isubFi}
\end{equation}%
holds for every $i\geq 1$. We will use induction by $i$. For $i=1$ (\ref%
{F*isubFi}) fulfills by definition of $F_{W}^{\ast }$. We assume that (\ref%
{F*isubFi}) proved for $i<j$. If $u\in \left( F_{W}^{\ast }\right) ^{r}$, $%
v\in \left( F_{W}^{\ast }\right) ^{s}$, such that $s+r=j$, $1\leq s+r$, then
by our assumption $u\in F^{r}$, $v\in F^{s}$. We have%
\begin{equation*}
u\times v=\alpha _{1,2}uv+\alpha _{2,1}vu+q_{3}\left( u,v\right) \in F^{j}.
\end{equation*}%
If $u,v\in F^{j}$ then%
\begin{equation*}
u\perp v=u+v+p_{2}\left( u,v\right) \in F^{j}.
\end{equation*}%
Also for every $\lambda \in k$%
\begin{equation*}
\lambda \ast u=\varphi \left( \lambda \right) u+l_{2}\left( u\right) \in
F^{j}.
\end{equation*}%
Therefore (\ref{F*isubFi}) is proved.

After this we can prove the inverse inclusion by \cite[Proposition 4.2]%
{TsurkovAutomEquiv}.
\end{proof}

\begin{proposition}
\label{sshomom}There is a homomorphism $\mathcal{S}_{n}:\mathfrak{S}%
_{n}\rightarrow \mathfrak{S}_{n-1}$.
\end{proposition}

\begin{proof}
We will define the mapping $\mathcal{S}_{n}:\mathfrak{S}_{n}\rightarrow 
\mathfrak{S}_{n-1}$ and after this we will prove that this mapping is a
homomorphism of groups. By Section \ref{operations} the element $\Phi \in 
\mathfrak{S}_{n}$ which is a strongly stable automorphism of the category%
\textit{\ }$\Theta _{n}^{0}$ can be defined by the system of words (\ref%
{words_list}) which fulfills conditions Op1) and Op20 as well as by the
system of bijections $\left\{ \sigma _{F^{\left( n\right) }}:F^{\left(
n\right) }\rightarrow F^{\left( n\right) }\mid F^{\left( n\right) }\in 
\mathrm{Ob}\Theta _{n}^{0}\right\} $ which fulfills conditions B1) and B2).
In this proof we actively use the second definition.

$F^{\left( n-1\right) }=F^{\left( n\right) }/\left( F^{\left( n\right)
}\right) ^{n-1}$ fulfills for every $F^{\left( n\right) }\in \mathrm{Ob}%
\Theta _{n}^{0}$, where $F^{\left( n-1\right) }\in \mathrm{Ob}\Theta
_{n-1}^{0}$. We consider $\Phi \in \mathfrak{S}_{n}$. By Section \ref%
{operations} $\sigma _{F^{\left( n\right) }}$ is an isomorphism $F^{\left(
n\right) }\rightarrow \left( F^{\left( n\right) }\right) _{W}^{\ast }$,
where $W=\left\{ \sigma _{F_{\omega }^{\left( n\right) }}\omega \mid \omega
\in \Omega \right\} $, $\Omega $ is a signature of linear algebras, $%
F_{\omega }^{\left( n\right) }=F^{\left( n\right) }\left( x_{1},\ldots
,x_{m}\right) \in \mathrm{Ob}\Theta _{n}^{0}$, $m$ is an arity of the
operation $\omega $. We denote $\kappa _{F^{\left( n\right) }}$ the
homomorphism $F^{\left( n\right) }\rightarrow F^{\left( n-1\right) }$. $%
F^{\left( n-1\right) }\in \Theta _{n}$, so in the set $F^{\left( n-1\right)
} $ we also can consider the verbal operations defined by the system of
words $W$. This algebra we denote $\left( F^{\left( n-1\right) }\right)
_{W}^{\ast } $. By \cite[Remark 3.1]{TsurkovAutomEquiv} $\kappa _{F^{\left(
n\right) }}$ is also homomorphism from $\left( F^{\left( n\right) }\right)
_{W}^{\ast }$ to $\left( F^{\left( n-1\right) }\right) _{W}^{\ast }$.

Now we will consider this diagram:%
\begin{equation}
\begin{array}{ccccccccc}
0 & \rightarrow  & \left( F^{\left( n\right) }\right) ^{n-1} & 
\hookrightarrow  & F^{\left( n\right) } & \underrightarrow{\kappa
_{F^{\left( n\right) }}} & F^{\left( n-1\right) } & \rightarrow  & 0 \\ 
&  & \sigma _{F^{\left( n\right) }}\downarrow  &  & \downarrow \sigma
_{F^{\left( n\right) }} &  & \downarrow \sigma _{F^{\left( n-1\right) }} & 
&  \\ 
0 & \rightarrow  & \left( F^{\left( n\right) }\right) ^{n-1} & 
\hookrightarrow  & \left( F^{\left( n\right) }\right) _{W}^{\ast } & 
\underrightarrow{\kappa _{F^{\left( n\right) }}} & \left( F^{\left(
n-1\right) }\right) _{W}^{\ast } & \rightarrow  & 0%
\end{array}%
.  \label{diagram}
\end{equation}%
Both rows of this diagram are exact because $\ker \kappa _{F^{\left(
n\right) }}=\left( F^{\left( n\right) }\right) ^{n-1}$. $\sigma _{F^{\left(
n\right) }}$is an isomorphism according the condition Op2). So 
\begin{equation}
\sigma _{F^{\left( n\right) }}\left( \left( F^{\left( n\right) }\right)
^{n-1}\right) =\left( \left( F^{\left( n\right) }\right) _{W}^{\ast }\right)
^{n-1}=\left( F^{\left( n\right) }\right) ^{n-1},  \label{surj_on_ker}
\end{equation}%
by Proposition \ref{F*i=Fi}. Therefore the left square of (\ref{diagram}) is
commutative. So we can close the right square of (\ref{diagram}) by uniquely
defined homomorphism $\sigma _{F^{\left( n-1\right) }}$ which fulfills 
\begin{equation}
\sigma _{F^{\left( n-1\right) }}\kappa _{F^{\left( n\right) }}=\kappa
_{F^{\left( n\right) }}\sigma _{F^{\left( n\right) }}.  \label{commut}
\end{equation}%
$\sigma _{F^{\left( n\right) }}$ is an isomorphism and (\ref{surj_on_ker})
fulfills, so $\sigma _{F^{\left( n-1\right) }}$ is an isomorphism.

Therefore for we have a system of bijections 
\begin{equation}
\left\{ \sigma _{F^{\left( n-1\right) }}:F^{\left( n-1\right) }\rightarrow
F^{\left( n-1\right) }\mid F^{\left( n-1\right) }\in \mathrm{Ob}\Theta
_{n-1}^{0}\right\} .  \label{bijectiobs}
\end{equation}%
We will check that (\ref{bijectiobs}) fulfills conditions B1) and B2). If $X$
is a set of free generators of $F^{\left( n\right) }\left( X\right)
=F^{\left( n\right) }\in \mathrm{Ob}\Theta _{n}^{0}$, then $\left\{ \kappa
_{F^{\left( n\right) }}\left( x\right) \mid x\in X\right\} $ is a set of
free generators of $F^{\left( n-1\right) }$. $\sigma _{F^{\left( n-1\right)
}}\kappa _{F^{\left( n\right) }}\left( x\right) =\kappa _{F^{\left( n\right)
}}\sigma _{F^{\left( n\right) }}\left( x\right) =\kappa _{F^{\left( n\right)
}}\left( x\right) $, so (\ref{bijectiobs}) fulfills condition B2).

We assume that $\psi :A^{\left( n-1\right) }\rightarrow B^{\left( n-1\right)
}\in \mathrm{Mor}\Theta _{n-1}^{0}$. $A^{\left( n-1\right) },B^{\left(
n-1\right) }\in \Theta _{n}$, $\kappa _{B^{\left( n\right) }}:B^{\left(
n\right) }\rightarrow B^{\left( n-1\right) }$ is an epimorphism, $A^{\left(
n\right) }$ is a free algebra in the variety $\Theta _{n}$, so there exists
a homomorphism $\widetilde{\psi }:A^{\left( n\right) }\rightarrow B^{\left(
n\right) }$, such that $\kappa _{B^{\left( n\right) }}\widetilde{\psi }=\psi
\kappa _{A^{\left( n\right) }}$. Therefore we have that 
\begin{equation*}
\sigma _{B^{\left( n-1\right) }}\psi \sigma _{A^{\left( n-1\right)
}}^{-1}\kappa _{A^{\left( n\right) }}=\sigma _{B^{\left( n-1\right) }}\psi
\kappa _{A^{\left( n\right) }}\sigma _{A^{\left( n\right) }}^{-1}=
\end{equation*}%
\begin{equation*}
\sigma _{B^{\left( n-1\right) }}\kappa _{B^{\left( n\right) }}\widetilde{%
\psi }\sigma _{A^{\left( n\right) }}^{-1}=\kappa _{B^{\left( n\right)
}}\sigma _{B^{\left( n\right) }}\widetilde{\psi }\sigma _{A^{\left( n\right)
}}^{-1}.
\end{equation*}%
The system of bijections $\left\{ \sigma _{F^{\left( n\right) }}\mid
F^{\left( n\right) }\in \mathrm{Ob}\Theta _{n}^{0}\right\} $ fulfills
condition B1), so $\sigma _{B^{\left( n\right) }}\widetilde{\psi }\sigma
_{A^{\left( n\right) }}^{-1}$ and $\kappa _{B^{\left( n\right) }}\sigma
_{B^{\left( n\right) }}\widetilde{\psi }\sigma _{A^{\left( n\right) }}^{-1}$
are homomorphisms. Therefore $\sigma _{B^{\left( n-1\right) }}\psi \sigma
_{A^{\left( n-1\right) }}^{-1}\kappa _{A^{\left( n\right) }}$ is a
homomorphism. If $\omega \in \Omega $ is a $m$-ary operation and $%
a_{1},\ldots ,a_{m}\in A^{\left( n-1\right) }$ then there are $f_{1},\ldots
,f_{m}\in A^{\left( n\right) }$ such that $a_{i}=\kappa _{A^{\left( n\right)
}}\left( f_{i}\right) $, $1\leq i\leq m$. So 
\begin{equation*}
\sigma _{B^{\left( n-1\right) }}\psi \sigma _{A^{\left( n-1\right)
}}^{-1}\omega \left( a_{1},\ldots ,a_{m}\right) =\sigma _{B^{\left(
n-1\right) }}\psi \sigma _{A^{\left( n-1\right) }}^{-1}\omega \left( \kappa
_{A^{\left( n\right) }}\left( f_{1}\right) ,\ldots ,\kappa _{A^{\left(
n\right) }}\left( f_{m}\right) \right) =
\end{equation*}%
\begin{equation*}
\sigma _{B^{\left( n-1\right) }}\psi \sigma _{A^{\left( n-1\right)
}}^{-1}\kappa _{A^{\left( n\right) }}\omega \left( f_{1},\ldots
,f_{m}\right) =
\end{equation*}%
\begin{equation*}
\omega \left( \sigma _{B^{\left( n-1\right) }}\psi \sigma _{A^{\left(
n-1\right) }}^{-1}\kappa _{A^{\left( n\right) }}\left( f_{1}\right) ,\ldots
,\sigma _{B^{\left( n-1\right) }}\psi \sigma _{A^{\left( n-1\right)
}}^{-1}\kappa _{A^{\left( n\right) }}\left( f_{m}\right) \right) =
\end{equation*}%
\begin{equation*}
\omega \left( \sigma _{B^{\left( n-1\right) }}\psi \sigma _{A^{\left(
n-1\right) }}^{-1}\left( a_{1}\right) ,\ldots ,\sigma _{B^{\left( n-1\right)
}}\psi \sigma _{A^{\left( n-1\right) }}^{-1}\left( a_{m}\right) \right) .
\end{equation*}%
Hence $\sigma _{B^{\left( n-1\right) }}\psi \sigma _{A^{\left( n-1\right)
}}^{-1}$ is a homomorphism. Analogously we can prove that $\sigma
_{B^{\left( n-1\right) }}^{-1}\psi \sigma _{A^{\left( n-1\right) }}$ is also
a homomorphism. Therefore (\ref{bijectiobs}) fulfills condition B1).

We can consider $\mathfrak{S}_{n}$ as set of bijections $\left\{ \sigma
_{F^{\left( n\right) }}:F^{\left( n\right) }\rightarrow F^{\left( n\right)
}\mid F^{\left( n\right) }\in \mathrm{Ob}\Theta _{n}^{0}\right\} $ which
fulfills conditions B1) and B2). And we can consider $\mathfrak{S}_{n-1}$ as
set of bijections (\ref{bijectiobs}) which fulfills conditions B1) and B2).
We prove that there is a mapping 
\begin{equation*}
\mathcal{S}_{n}:\mathfrak{S}_{n}\ni \left\{ \sigma _{F^{\left( n\right)
}}:F^{\left( n\right) }\rightarrow F^{\left( n\right) }\mid F^{\left(
n\right) }\in \mathrm{Ob}\Theta _{n}^{0}\right\} \rightarrow
\end{equation*}%
\begin{equation*}
\left\{ \sigma _{F^{\left( n-1\right) }}:F^{\left( n-1\right) }\rightarrow
F^{\left( n-1\right) }\mid F^{\left( n-1\right) }\in \mathrm{Ob}\Theta
_{n-1}^{0}\right\} \in \mathfrak{S}_{n-1}\text{.}
\end{equation*}%
Now we will prove that this mapping is a homomorphism. We also will use
notation $\mathcal{S}_{n}\left( \sigma _{F^{\left( n\right) }}\right)
=\sigma _{F^{\left( n-1\right) }}$, where $\sigma _{F^{\left( n-1\right) }}$
defined by condition (\ref{commut}). We assume that $\Phi _{1},\Phi _{2}\in 
\mathfrak{S}_{n}$. $\Phi _{i}$ is defined by system of bijection $\left\{
\sigma _{F^{\left( n\right) }}^{\left( i\right) }\mid F^{\left( n\right)
}\in \mathrm{Ob}\Theta _{n}^{0}\right\} $, $i=1,2$. $\mathcal{S}_{n}\left(
\Phi _{i}\right) $ is defined by system of bijection $\left\{ \mathcal{S}%
_{n}\left( \sigma _{F^{\left( n\right) }}^{\left( i\right) }\right) \mid
F^{\left( n\right) }\in \mathrm{Ob}\Theta _{n}^{0}\right\} $, $i=1,2$. $\Phi
_{1}\Phi _{2}$ is defined by system of bijection $\left\{ \sigma _{F^{\left(
n\right) }}^{\left( 1\right) }\sigma _{F^{\left( n\right) }}^{\left(
2\right) }\mid F^{\left( n\right) }\in \mathrm{Ob}\Theta _{n}^{0}\right\} $
and $\mathcal{S}_{n}\left( \Phi _{1}\Phi _{2}\right) $ is defined by system
of bijection $\left\{ \mathcal{S}_{n}\left( \sigma _{F^{\left( n\right)
}}^{\left( 1\right) }\sigma _{F^{\left( n\right) }}^{\left( 2\right)
}\right) \mid F^{\left( n\right) }\in \mathrm{Ob}\Theta _{n}^{0}\right\} $.
By (\ref{commut}) we have that%
\begin{equation*}
\mathcal{S}_{n}\left( \sigma _{F^{\left( n\right) }}^{\left( 1\right)
}\right) \mathcal{S}_{n}\left( \sigma _{F^{\left( n\right) }}^{\left(
2\right) }\right) \kappa _{F^{\left( n\right) }}=\mathcal{S}_{n}\left(
\sigma _{F^{\left( n\right) }}^{\left( 1\right) }\right) \kappa _{F^{\left(
n\right) }}\sigma _{F^{\left( n\right) }}^{\left( 2\right) }=\kappa
_{F^{\left( n\right) }}\left( \sigma _{F^{\left( n\right) }}^{\left(
1\right) }\sigma _{F^{\left( n\right) }}^{\left( 2\right) }\right)
\end{equation*}%
fulfills for every $F^{\left( n\right) }\in \mathrm{Ob}\Theta _{n}^{0}$.
Therefore 
\begin{equation*}
\mathcal{S}_{n}\left( \sigma _{F^{\left( n\right) }}^{\left( 1\right)
}\right) \mathcal{S}_{n}\left( \sigma _{F^{\left( n\right) }}^{\left(
2\right) }\right) =\mathcal{S}_{n}\left( \sigma _{F^{\left( n\right)
}}^{\left( 1\right) }\sigma _{F^{\left( n\right) }}^{\left( 2\right) }\right)
\end{equation*}%
and $\mathcal{S}_{n}\left( \Phi _{1}\right) \mathcal{S}_{n}\left( \Phi
_{2}\right) =\mathcal{S}_{n}\left( \Phi _{1}\Phi _{2}\right) $.
\end{proof}

\begin{corollary}
\label{rez}For every $\Phi \in \mathfrak{S}_{n}$ which defined by system of
words $W=\left\{ w_{\omega ,n}\mid \omega \in \Omega \right\} $ there exists 
$\Psi \in \mathfrak{S}_{n-1}$ which defined by system of words $\widetilde{W}%
=\left\{ \kappa _{F_{\omega }^{\left( n\right) }}w_{\omega ,n-1}\mid \omega
\in \Omega \right\} $ such that $w_{\omega ,n}=w_{\omega ,n-1}+r_{\omega ,n}$%
, $w_{\omega ,n}\in F_{m_{\omega }}^{\left( n\right) }$, $w_{\omega ,n-1}\in
\bigoplus\limits_{i=1}^{n-2}\left( F_{m_{\omega }}^{\left( n\right) }\right)
_{i}$, $r_{\omega ,n}\in \left( F_{m_{\omega }}^{\left( n\right) }\right)
_{n-1}$, $F_{m_{\omega }}^{\left( n\right) }=F^{\left( n\right) }\left(
x_{1},\ldots ,x_{m_{\omega }}\right) \in \mathrm{Ob}\Theta _{i}^{0}$, $%
m_{\omega }$ is arity of $\omega $.
\end{corollary}

\begin{proof}
We take $\Psi =\mathcal{S}_{n}\left( \Phi \right) $. The system of bijection 
$\left\{ \sigma _{F^{\left( n\right) }}\mid F^{\left( n\right) }\in \mathrm{%
Ob}\Theta _{n}^{0}\right\} $ corresponds to the automorphism $\Phi $, the
system of bijection $\left\{ \sigma _{F^{\left( n-1\right) }}\mid F^{\left(
n-1\right) }\in \mathrm{Ob}\Theta _{n-1}^{0}\right\} $ corresponds to the
automorphism $\Psi $. The system of words $\left\{ \sigma _{F^{\left(
n-1\right) }}\kappa _{F_{\omega }^{\left( n\right) }}\omega \mid \omega \in
\Omega \right\} $ defines the automorphism $\Psi $. For every $\omega \in
\Omega $ we can decompose $\sigma _{F^{\left( n-1\right) }}\omega =w_{\omega
,n}=w_{\omega ,n-1}+r_{\omega ,n}$. We have that 
\begin{equation*}
\kappa _{F_{\omega }^{\left( n\right) }}w_{\omega ,n-1}=\kappa _{F_{\omega
}^{\left( n\right) }}w_{\omega ,n}=\kappa _{F_{\omega }^{\left( n\right)
}}\sigma _{F^{\left( n-1\right) }}\omega =\sigma _{F^{\left( n-1\right)
}}\kappa _{F_{\omega }^{\left( n\right) }}\omega .
\end{equation*}
\end{proof}

\setcounter{corollary}{0}

\begin{proposition}
\label{i_homom}If $\Phi \in \mathfrak{S}_{n}\mathfrak{\cap Y}_{n}$, then $%
\mathcal{S}_{n}\left( \Phi \right) \in \mathfrak{S}_{n-1}\mathfrak{\cap Y}%
_{n-1}$.
\end{proposition}

\begin{proof}
If $\Phi \in \mathfrak{S}_{n}\mathfrak{\cap Y}_{n}$, then by definitions \ref%
{inner} and \ref{str_stab_aut} there exists a system of isomorphisms $%
\left\{ s_{F^{\left( n\right) }}:F^{\left( n\right) }\rightarrow F^{\left(
n\right) }\mid F^{\left( n\right) }\in \mathrm{Ob}\Theta _{n}^{0}\right\} $
such that $s_{B^{\left( n\right) }}\psi =\Phi \left( \psi \right)
s_{A^{\left( n\right) }}$ fulfills for every $\psi \in \mathrm{Mor}_{\Theta
_{n}^{0}}\left( A^{\left( n\right) },B^{\left( n\right) }\right) $. For
every $F^{\left( n\right) }\in \mathrm{Ob}\Theta _{n}^{0}$ we consider the
diagram%
\begin{equation*}
\begin{array}{ccccccccc}
0 & \rightarrow  & \left( F^{\left( n\right) }\right) ^{n-1} & 
\hookrightarrow  & F^{\left( n\right) } & \underrightarrow{\kappa
_{F^{\left( n\right) }}} & F^{\left( n-1\right) } & \rightarrow  & 0 \\ 
&  & s_{F^{\left( n\right) }}\downarrow  &  & \downarrow s_{F^{\left(
n\right) }} &  & \downarrow s_{F^{\left( n-1\right) }} &  &  \\ 
0 & \rightarrow  & \left( F^{\left( n\right) }\right) ^{n-1} & 
\hookrightarrow  & F^{\left( n\right) } & \underrightarrow{\kappa
_{F^{\left( n\right) }}} & F^{\left( n-1\right) } & \rightarrow  & 0%
\end{array}%
\end{equation*}%
and simpler than after the consideration of the diagram (\ref{diagram})
conclude that we can close the right square of this diagram by isomorphism $%
s_{F^{\left( n-1\right) }}$ which fulfills $s_{F^{\left( n-1\right) }}\kappa
_{F^{\left( n\right) }}=\kappa _{F^{\left( n\right) }}s_{F^{\left( n\right)
}}$.

We will consider a homomorphism $\psi \in \mathrm{Mor}_{\Theta
_{n-1}^{0}}\left( A^{\left( n-1\right) },B^{\left( n-1\right) }\right) $. As
in the proof of the Proposition \ref{sshomom} we can conclude that there
exists a homomorphism $\widetilde{\psi }\in \mathrm{Mor}_{\Theta
_{n}^{0}}\left( A^{\left( n\right) },B^{\left( n\right) }\right) $ such that 
$\kappa _{B^{\left( n\right) }}\widetilde{\psi }=\psi \kappa _{A^{\left(
n\right) }}$ fulfills. The system of bijections which corresponds to the
automorhism $\Phi $ we denote by $\left\{ \sigma _{F^{\left( n\right) }}\mid
F^{\left( n\right) }\in \mathrm{Ob}\Theta _{n}^{0}\right\} $ and the system
of bijections which corresponds to the automorphism $\mathcal{S}_{n}\left(
\Phi \right) $ we denote by $\left\{ \mathcal{S}_{n}\left( \sigma
_{F^{\left( n\right) }}\right) :F^{\left( n-1\right) }\rightarrow F^{\left(
n-1\right) }\mid F^{\left( n\right) }\in \mathrm{Ob}\Theta _{n}^{0}\right\} $%
. By definition \ref{str_stab_aut} the $\Phi \left( \widetilde{\psi }\right)
=\sigma _{B^{\left( n\right) }}\widetilde{\psi }\sigma _{A^{\left( n\right)
}}^{-1}$ and $\left( \mathcal{S}_{n}\left( \Phi \right) \right) \left( \psi
\right) =\mathcal{S}_{n}\left( \sigma _{B^{\left( n\right) }}\right) \psi
\left( \mathcal{S}_{n}\left( \sigma _{A^{\left( n\right) }}\right) \right)
^{-1}$ fulfills. So we have that 
\begin{equation*}
\left( \left( \mathcal{S}_{n}\left( \Phi \right) \right) \left( \psi \right)
\right) s_{A^{\left( n-1\right) }}\kappa _{A^{\left( n\right) }}=\mathcal{S}%
_{n}\left( \sigma _{B^{\left( n\right) }}\right) \psi \left( \mathcal{S}%
_{n}\left( \sigma _{A^{\left( n\right) }}\right) \right) ^{-1}\kappa
_{A^{\left( n\right) }}s_{A^{\left( n\right) }}=
\end{equation*}%
\begin{equation*}
\mathcal{S}_{n}\left( \sigma _{B^{\left( n\right) }}\right) \psi \kappa
_{A^{\left( n\right) }}\sigma _{A^{\left( n\right) }}^{-1}s_{A^{\left(
n\right) }}=\mathcal{S}_{n}\left( \sigma _{B^{\left( n\right) }}\right)
\kappa _{B^{\left( n\right) }}\widetilde{\psi }\sigma _{A^{\left( n\right)
}}^{-1}s_{A^{\left( n\right) }}=
\end{equation*}%
\begin{equation*}
\kappa _{B^{\left( n\right) }}\sigma _{B^{\left( n\right) }}\widetilde{\psi }%
\sigma _{A^{\left( n\right) }}^{-1}s_{A^{\left( n\right) }}=\kappa
_{B^{\left( n\right) }}\Phi \left( \widetilde{\psi }\right) s_{A^{\left(
n\right) }}=
\end{equation*}%
\begin{equation*}
\kappa _{B^{\left( n\right) }}s_{B^{\left( n\right) }}\widetilde{\psi }%
=s_{B^{\left( n-1\right) }}\kappa _{B^{\left( n\right) }}\widetilde{\psi }%
=s_{B^{\left( n-1\right) }}\psi \kappa _{A^{\left( n\right) }}
\end{equation*}%
holds. Therefore $\left( \left( \mathcal{S}_{n}\left( \Phi \right) \right)
\left( \psi \right) \right) s_{A^{\left( n-1\right) }}=s_{B^{\left(
n-1\right) }}\psi $ and $\mathcal{S}_{n}\left( \Phi \right) $ is an inner
automorphism with the system of isomorphisms $\left\{ s_{F^{\left(
n-1\right) }}\mid F^{\left( n-1\right) }\in \mathrm{Ob}\Theta
_{n-1}^{0}\right\} $.
\end{proof}

From Propositions \ref{sshomom} and \ref{i_homom} we can conclude the

\begin{theorem}
There is a homomorphism $\widetilde{\mathcal{S}}_{n}:\mathfrak{A}_{n}%
\mathfrak{/Y}_{n}\rightarrow \mathfrak{A}_{n-1}\mathfrak{/Y}_{n-1}$.
\end{theorem}

\subsection{Group $\mathfrak{A}_{3}\mathfrak{/Y}_{3}$.\label{nilp3}}

\begin{theorem}
$\mathfrak{A}_{3}\mathfrak{/Y}_{3}\cong k^{\ast }\mathfrak{\leftthreetimes }%
\mathrm{Aut}k$.
\end{theorem}

\begin{proof}
We assume that the system of words $W$ fulfills conditions Op1) and Op2) in
the variety $\Theta _{3}$. We will find the specific form of these words. As
in the proof of the Proposition \ref{F*i=Fi} we have that $w_{0,3}=0$, $%
w_{+,3}\left( x_{1},x_{2}\right) =x_{1}+x_{2}+p_{2}\left( x_{1},x_{2}\right) 
$, where $p_{2}\left( x_{1},x_{2}\right) \in \left( F^{\left( 3\right)
}\left( x_{1},x_{2}\right) \right) ^{2}$ and $p_{2}\left( x_{1},0\right)
=p_{2}\left( 0,x_{2}\right) =0$ fulfills. Therefore $p_{2}\left(
x_{1},x_{2}\right) =\gamma _{1,2}x_{1}x_{2}+\gamma _{2,1}x_{2}x_{1}$, where $%
\gamma _{1,2},\gamma _{2,1}\in k$. By Op2) $w_{+,3}\left( x_{1},x_{2}\right)
=w_{+,3}\left( x_{2},x_{1}\right) $ must fulfill, so $\gamma _{1,2}=\gamma
_{2,1}$.

By Op1) $w_{\lambda ,3}\left( x\right) =\varphi \left( \lambda \right)
x+\psi \left( \lambda \right) x^{2}\in F\left( x\right) $, where $\varphi
\left( \lambda \right) ,\psi \left( \lambda \right) \in k$, fulfills for
every $\lambda \in k$. $\lambda \ast \left( x_{1}\perp x_{2}\right) =\left(
\lambda \ast x_{1}\right) \perp \left( \lambda \ast x_{2}\right) $ must
fulfill in $F^{\left( 3\right) }\left( x_{1},x_{2}\right) $. 
\begin{equation*}
\lambda \ast \left( x_{1}\perp x_{2}\right) =\lambda \ast \left(
x_{1}+x_{2}+\gamma _{1,2}x_{1}x_{2}+\gamma _{1,2}x_{2}x_{1}\right) =
\end{equation*}%
\begin{equation*}
\varphi \left( \lambda \right) \left( x_{1}+x_{2}+\gamma
_{1,2}x_{1}x_{2}+\gamma _{1,2}x_{2}x_{1}\right) +\psi \left( \lambda \right)
\left( x_{1}+x_{2}+\gamma _{1,2}x_{1}x_{2}+\gamma _{1,2}x_{2}x_{1}\right)
^{2}=
\end{equation*}%
\begin{equation*}
\varphi \left( \lambda \right) x_{1}+\varphi \left( \lambda \right)
x_{2}+\varphi \left( \lambda \right) \gamma _{1,2}x_{1}x_{2}+\varphi \left(
\lambda \right) \gamma _{1,2}x_{2}x_{1}+
\end{equation*}%
\begin{equation*}
\psi \left( \lambda \right) x_{1}^{2}+\psi \left( \lambda \right)
x_{2}^{2}+\psi \left( \lambda \right) x_{1}x_{2}+\psi \left( \lambda \right)
x_{2}x_{1}.
\end{equation*}%
\begin{equation*}
\left( \lambda \ast x_{1}\right) \perp \left( \lambda \ast x_{2}\right)
=\left( \varphi \left( \lambda \right) x_{1}+\psi \left( \lambda \right)
x_{1}^{2}\right) \perp \left( \varphi \left( \lambda \right) x_{2}+\psi
\left( \lambda \right) x_{2}^{2}\right) =
\end{equation*}%
\begin{equation*}
\varphi \left( \lambda \right) x_{1}+\psi \left( \lambda \right)
x_{1}^{2}+\varphi \left( \lambda \right) x_{2}+\psi \left( \lambda \right)
x_{2}^{2}+
\end{equation*}%
\begin{equation*}
\gamma _{1,2}\left( \varphi \left( \lambda \right) x_{1}+\psi \left( \lambda
\right) x_{1}^{2}\right) \left( \varphi \left( \lambda \right) x_{2}+\psi
\left( \lambda \right) x_{2}^{2}\right) +
\end{equation*}%
\begin{equation*}
\gamma _{1,2}\left( \varphi \left( \lambda \right) x_{2}+\psi \left( \lambda
\right) x_{2}^{2}\right) \left( \varphi \left( \lambda \right) x_{1}+\psi
\left( \lambda \right) x_{1}^{2}\right) =
\end{equation*}%
\begin{equation*}
\varphi \left( \lambda \right) x_{1}+\psi \left( \lambda \right)
x_{1}^{2}+\varphi \left( \lambda \right) x_{2}+\psi \left( \lambda \right)
x_{2}^{2}+\gamma _{1,2}\left( \varphi \left( \lambda \right) \right)
^{2}x_{1}x_{2}+\gamma _{1,2}\left( \varphi \left( \lambda \right) \right)
^{2}x_{2}x_{1}.
\end{equation*}%
So $\varphi \left( \lambda \right) \gamma _{1,2}+\psi \left( \lambda \right)
=\gamma _{1,2}\left( \varphi \left( \lambda \right) \right) ^{2}$ and $\psi
\left( \lambda \right) =\gamma _{1,2}\left( \left( \varphi \left( \lambda
\right) \right) ^{2}-\varphi \left( \lambda \right) \right) $.

$\left( \mu +\lambda \right) \ast x=\left( \mu \ast x\right) \perp \left(
\lambda \ast x\right) $ must fulfill in $F^{\left( 3\right) }\left( x\right) 
$ for every $\mu ,\lambda \in k$.%
\begin{equation*}
\left( \mu +\lambda \right) \ast x\equiv \varphi \left( \mu +\lambda \right)
x\left( \func{mod}\left( F^{\left( 3\right) }\left( x\right) \right)
^{2}\right)
\end{equation*}%
\begin{equation*}
\left( \mu \ast x\right) \perp \left( \lambda \ast x\right) \equiv \left(
\varphi \left( \mu \right) +\varphi \left( \lambda \right) \right) x\left( 
\func{mod}\left( F^{\left( 3\right) }\left( x\right) \right) ^{2}\right) ,
\end{equation*}%
so $\varphi \left( \mu +\lambda \right) =\varphi \left( \mu \right) +\varphi
\left( \lambda \right) $. From $\left( \mu \lambda \right) \ast x=\mu \ast
\left( \lambda \ast x\right) $ we conclude that $\varphi \left( \mu \lambda
\right) =\varphi \left( \mu \right) \varphi \left( \lambda \right) $. We
consider $\mu \in k$. By condition Op1) $\sigma _{F^{\left( 3\right) }\left(
x\right) }:F^{\left( 3\right) }\left( x\right) \rightarrow \left( F^{\left(
3\right) }\left( x\right) \right) _{W}^{\ast }$ is an isomorphism. So $\mu
x=\lambda \ast x\perp \nu \ast \left( x\times x\right) $, where $\lambda
,\nu \in k$.

As in the proof of the Proposition \ref{F*i=Fi} we have that $w_{\cdot
}\left( x_{1},x_{2}\right) =\alpha _{1,2}x_{1}x_{2}+\alpha _{2,1}x_{2}x_{1}$.

By Proposition \ref{F*i=Fi} we have that $\mu x\equiv \lambda \ast x\equiv
\varphi \left( \lambda \right) x\left( \func{mod}\left( F^{\left( 3\right)
}\left( x\right) \right) ^{2}\right) $, so $\mu =\varphi \left( \lambda
\right) $. Therefore $\varphi \in \mathrm{Aut}k$.

We prove that if the system of words (\ref{words_list}) fulfills conditions
Op1) and Op2) then necessary%
\begin{equation*}
w_{0,3}=0,w_{+,3}\left( x_{1},x_{2}\right) =x_{1}+x_{2}+\gamma
_{1,2}x_{1}x_{2}+\gamma _{1,2}x_{2}x_{1},
\end{equation*}%
\begin{equation}
w_{\lambda ,3}\left( x_{1}\right) =\varphi \left( \lambda \right)
x_{1}+\gamma _{1,2}\left( \left( \varphi \left( \lambda \right) \right)
^{2}-\varphi \left( \lambda \right) \right) x_{1}^{2}\left( \lambda \in
k\right) ,  \label{real_list_nilp3}
\end{equation}%
\begin{equation*}
w_{\cdot ,3}\left( x_{1},x_{2}\right) =\alpha _{1,2}x_{1}x_{2}+\alpha
_{2,1}x_{2}x_{1},
\end{equation*}%
where $\alpha _{1,2},\alpha _{2,1},\gamma _{1,2}\in k$, $\varphi \in \mathrm{%
Aut}k$. If we will construct the isomorphism $\sigma _{F}:F\rightarrow
F_{W}^{\ast }$ for every $F\in \mathrm{Ob}\Theta _{3}^{0}$, then, as in the 
\cite{Tsurkov} we must demand $\alpha _{1,2}\neq \pm \alpha _{2,1}$.

Now we will prove that if the system of words $W$ has form (\ref%
{real_list_nilp3}) then this system of words fulfills conditions Op1) and
Op2). First of all we must check that if $H\in \Theta _{3}$ and the system
of words $W$ has form (\ref{real_list_nilp3}) then $H_{W}^{\ast }\in \Theta
_{3}$. We can check it by direct calculation. For example, if $h\in H\in
\Theta _{3}$, $\mu ,\lambda \in k$ then 
\begin{equation*}
\left( \mu \lambda \right) \ast h=\varphi \left( \mu \lambda \right)
h+\gamma _{1,2}\left( \left( \varphi \left( \mu \lambda \right) \right)
^{2}-\varphi \left( \mu \lambda \right) \right) h^{2},
\end{equation*}%
\begin{equation*}
\mu \ast \left( \lambda \ast h\right) =\mu \ast \left( \varphi \left(
\lambda \right) h+\gamma _{1,2}\left( \left( \varphi \left( \lambda \right)
\right) ^{2}-\varphi \left( \lambda \right) \right) h^{2}\right) =
\end{equation*}%
\begin{equation*}
\varphi \left( \mu \right) \left( \varphi \left( \lambda \right) h+\gamma
_{1,2}\left( \left( \varphi \left( \lambda \right) \right) ^{2}-\varphi
\left( \lambda \right) \right) h^{2}\right) +\gamma _{1,2}\left( \left(
\varphi \left( \mu \right) \right) ^{2}-\varphi \left( \mu \right) \right)
\left( \varphi \left( \lambda \right) \right) ^{2}h^{2}=
\end{equation*}%
\begin{equation*}
\varphi \left( \mu \right) \varphi \left( \lambda \right) h+\gamma
_{1,2}\left( \varphi \left( \mu \right) \left( \left( \varphi \left( \lambda
\right) \right) ^{2}-\varphi \left( \lambda \right) \right) +\left( \left(
\varphi \left( \mu \right) \right) ^{2}-\varphi \left( \mu \right) \right)
\left( \varphi \left( \lambda \right) \right) ^{2}\right) h^{2}=
\end{equation*}%
\begin{equation*}
\varphi \left( \mu \right) \varphi \left( \lambda \right) h+\gamma
_{1,2}\left( \left( \varphi \left( \mu \right) \right) ^{2}\left( \varphi
\left( \lambda \right) \right) ^{2}-\varphi \left( \mu \right) \varphi
\left( \lambda \right) \right) h^{2}.
\end{equation*}%
If $h_{1},h_{2}\in H\in \Theta _{3}$ then $\lambda \in k$%
\begin{equation*}
\left( \lambda \ast h_{1}\right) \times h_{2}=
\end{equation*}%
\begin{equation*}
\left( \varphi \left( \lambda \right) h_{1}+\gamma _{1,2}\left( \left(
\varphi \left( \lambda \right) \right) ^{2}-\varphi \left( \lambda \right)
\right) h_{1}^{2}\right) \times h_{2}=
\end{equation*}%
\begin{equation*}
\alpha _{1,2}\varphi \left( \lambda \right) h_{1}h_{2}+\alpha _{2,1}\varphi
\left( \lambda \right) h_{2}h_{1},
\end{equation*}%
\begin{equation*}
\lambda \ast \left( h_{1}\times h_{2}\right) =
\end{equation*}%
\begin{equation*}
\lambda \ast \left( \alpha _{1,2}h_{1}h_{2}+\alpha _{2,1}h_{2}h_{1}\right) =
\end{equation*}%
\begin{equation*}
\varphi \left( \lambda \right) \left( \alpha _{1,2}h_{1}h_{2}+\alpha
_{2,1}h_{2}h_{1}\right) .
\end{equation*}%
By similar way we can prove that $h_{1}\times \left( \lambda \ast
h_{2}\right) =\lambda \ast \left( h_{1}\times h_{2}\right) $. Other axioms
of the variety $\Theta _{3}$ fulfill in $H_{W}^{\ast }$ by the constructions
of the words (\ref{real_list_nilp3}).

It means that if the system of words $W$ has form (\ref{real_list_nilp3})
then for every $F=F\left( X\right) \in \mathrm{Ob}\Theta _{3}^{0}$ exists a
homomorphism $\sigma _{F}:F\rightarrow F_{W}^{\ast }$ such that $\sigma
_{F}\mid _{X}=id_{X}$. Our goal is to prove that these homomorphisms are
isomorphisms. For this purpose we will research the superpositions of these
homomorphisms. If the system of words $W$ has form (\ref{real_list_nilp3})
then it depends on the parameters $\varphi $, $\gamma _{1,2}$, $\alpha _{1,2}
$ and $\alpha _{2,1}$, where $\varphi \in \mathrm{Aut}k$, $\gamma
_{1,2},\alpha _{1,2},\alpha _{2,1}\in k$, $\alpha _{1,2}\neq \pm \alpha
_{2,1}$. We will denote the homomorphism $\sigma _{F}$ which corresponds to
the system of words $W$ with parameters $\varphi $, $\gamma _{1,2}$, $\alpha
_{1,2}$, $\alpha _{2,1}$ by $\sigma _{F}=\sigma _{F}\left( \varphi ,\gamma
_{1,2},\alpha _{1,2},\alpha _{2,1}\right) $, or for shortness, $\sigma
_{F}\left( \varphi ,\gamma _{1,2},\alpha _{1,2},\alpha _{2,1}\right) $. It
is clear that $\sigma _{F}\left( id_{k},0,1,0\right) =id_{F}$ for every $%
F\in \mathrm{Ob}\Theta _{3}^{0}$. We consider two system of of words $%
W^{\left( i\right) }$, $i=1,2$. Both these systems have form (\ref%
{real_list_nilp3}) and defined the system of homomorphisms $\left\{ \sigma
_{F}\left( \varphi ^{\left( i\right) },\gamma _{1,2}^{\left( i\right)
},\alpha _{1,2}^{\left( i\right) },\alpha _{2,1}^{\left( i\right) }\right)
\mid F\in \mathrm{Ob}\Theta _{3}^{0}\right\} $, $i=1,2$. We have in $F\left(
x\right) \in \mathrm{Ob}\Theta _{3}^{0}$ for every $\lambda \in k$%
\begin{equation*}
\sigma \left( \varphi ^{\left( 2\right) },\gamma _{1,2}^{\left( 2\right)
},\alpha _{1,2}^{\left( 2\right) },\alpha _{2,1}^{\left( 2\right) }\right)
\sigma \left( \varphi ^{\left( 1\right) },\gamma _{1,2}^{\left( 1\right)
},\alpha _{1,2}^{\left( 1\right) },\alpha _{2,1}^{\left( 1\right) }\right)
\left( \lambda x\right) =
\end{equation*}%
\begin{equation*}
\sigma \left( \varphi ^{\left( 2\right) },\gamma _{1,2}^{\left( 2\right)
},\alpha _{1,2}^{\left( 2\right) },\alpha _{2,1}^{\left( 2\right) }\right)
\left( \varphi ^{\left( 1\right) }\left( \lambda \right) x+\gamma
_{1,2}^{\left( 1\right) }\left( \left( \varphi ^{\left( 1\right) }\left(
\lambda \right) \right) ^{2}-\varphi ^{\left( 1\right) }\left( \lambda
\right) \right) x^{2}\right) =
\end{equation*}%
\begin{equation*}
\varphi ^{\left( 1\right) }\left( \lambda \right) \underset{\left( 2\right) }%
{\ast }x\underset{\left( 2\right) }{\perp }\gamma _{1,2}^{\left( 1\right)
}\left( \left( \varphi ^{\left( 1\right) }\left( \lambda \right) \right)
^{2}-\varphi ^{\left( 1\right) }\left( \lambda \right) \right) \underset{%
\left( 2\right) }{\ast }\left( x\underset{\left( 2\right) }{\times }x\right)
=
\end{equation*}%
\begin{equation*}
\left( \varphi ^{\left( 2\right) }\left( \varphi ^{\left( 1\right) }\left(
\lambda \right) \right) x+\gamma _{1,2}^{\left( 2\right) }\left( \left(
\varphi ^{\left( 2\right) }\left( \varphi ^{\left( 1\right) }\left( \lambda
\right) \right) \right) ^{2}-\varphi ^{\left( 2\right) }\left( \varphi
^{\left( 1\right) }\left( \lambda \right) \right) \right) x^{2}\right) 
\underset{\left( 2\right) }{\perp }
\end{equation*}%
\begin{equation*}
\varphi ^{\left( 2\right) }\left( \gamma _{1,2}^{\left( 1\right) }\left(
\left( \varphi ^{\left( 1\right) }\left( \lambda \right) \right)
^{2}-\varphi ^{\left( 1\right) }\left( \lambda \right) \right) \right)
\left( \alpha _{1,2}^{\left( 2\right) }+\alpha _{2,1}^{\left( 2\right)
}\right) x^{2}=
\end{equation*}%
\begin{equation*}
\left( \varphi ^{\left( 2\right) }\varphi ^{\left( 1\right) }\right) \left(
\lambda \right) x+
\end{equation*}%
\begin{equation*}
\left( \gamma _{1,2}^{\left( 2\right) }+\varphi ^{\left( 2\right) }\left(
\gamma _{1,2}^{\left( 1\right) }\right) \left( \alpha _{1,2}^{\left(
2\right) }+\alpha _{2,1}^{\left( 2\right) }\right) \right) \left( \left(
\left( \varphi ^{\left( 2\right) }\varphi ^{\left( 1\right) }\right) \left(
\lambda \right) \right) ^{2}-\left( \varphi ^{\left( 2\right) }\varphi
^{\left( 1\right) }\right) \left( \lambda \right) \right) x^{2}.
\end{equation*}%
Also we have in $F\left( x_{1},x_{2}\right) \in \mathrm{Ob}\Theta _{3}^{0}$%
\begin{equation*}
\sigma \left( \varphi ^{\left( 2\right) },\gamma _{1,2}^{\left( 2\right)
},\alpha _{1,2}^{\left( 2\right) },\alpha _{2,1}^{\left( 2\right) }\right)
\sigma \left( \varphi ^{\left( 1\right) },\gamma _{1,2}^{\left( 1\right)
},\alpha _{1,2}^{\left( 1\right) },\alpha _{2,1}^{\left( 1\right) }\right)
\left( x_{1}+x_{2}\right) =
\end{equation*}%
\begin{equation*}
\sigma \left( \varphi ^{\left( 2\right) },\gamma _{1,2}^{\left( 2\right)
},\alpha _{1,2}^{\left( 2\right) },\alpha _{2,1}^{\left( 2\right) }\right)
\left( x_{1}+x_{2}+\gamma _{1,2}^{\left( 1\right) }x_{1}x_{2}+\gamma
_{1,2}^{\left( 1\right) }x_{2}x_{1}\right) =
\end{equation*}%
\begin{equation*}
x_{1}\underset{\left( 2\right) }{\perp }x_{2}\underset{\left( 2\right) }{%
\perp }\gamma _{1,2}^{\left( 1\right) }\underset{\left( 2\right) }{\ast }%
\left( x_{1}\underset{\left( 2\right) }{\times }x_{2}\right) \underset{%
\left( 2\right) }{\perp }\gamma _{1,2}^{\left( 1\right) }\underset{\left(
2\right) }{\ast }\left( x_{2}\underset{\left( 2\right) }{\times }%
x_{1}\right) =
\end{equation*}%
\begin{equation*}
\left( x_{1}+x_{2}+\gamma _{1,2}^{\left( 2\right) }x_{1}x_{2}+\gamma
_{1,2}^{\left( 2\right) }x_{2}x_{1}\right) +
\end{equation*}%
\begin{equation*}
\varphi ^{\left( 2\right) }\left( \gamma _{1,2}^{\left( 1\right) }\right)
\left( \alpha _{1,2}^{\left( 2\right) }x_{1}x_{2}+\alpha _{2,1}^{\left(
2\right) }x_{2}x_{1}\right) +\varphi ^{\left( 2\right) }\left( \gamma
_{1,2}^{\left( 1\right) }\right) \left( \alpha _{1,2}^{\left( 2\right)
}x_{2}x_{1}+\alpha _{2,1}^{\left( 2\right) }x_{1}x_{2}\right) =
\end{equation*}%
\begin{equation*}
x_{1}+x_{2}+\left( \gamma _{1,2}^{\left( 2\right) }+\varphi ^{\left(
2\right) }\left( \gamma _{1,2}^{\left( 1\right) }\right) \left( \alpha
_{1,2}^{\left( 2\right) }+\alpha _{2,1}^{\left( 2\right) }\right) \right)
x_{1}x_{2}+
\end{equation*}%
\begin{equation*}
\left( \gamma _{1,2}^{\left( 2\right) }+\varphi ^{\left( 2\right) }\left(
\gamma _{1,2}^{\left( 1\right) }\right) \left( \alpha _{1,2}^{\left(
2\right) }+\alpha _{2,1}^{\left( 2\right) }\right) \right) x_{2}x_{1}.
\end{equation*}%
\begin{equation*}
\sigma \left( \varphi ^{\left( 2\right) },\gamma _{1,2}^{\left( 2\right)
},\alpha _{1,2}^{\left( 2\right) },\alpha _{2,1}^{\left( 2\right) }\right)
\sigma \left( \varphi ^{\left( 1\right) },\gamma _{1,2}^{\left( 1\right)
},\alpha _{1,2}^{\left( 1\right) },\alpha _{2,1}^{\left( 1\right) }\right)
\left( x_{1}x_{2}\right) =
\end{equation*}%
\begin{equation*}
\sigma \left( \varphi ^{\left( 2\right) },\gamma _{1,2}^{\left( 2\right)
},\alpha _{1,2}^{\left( 2\right) },\alpha _{2,1}^{\left( 2\right) }\right)
\left( \alpha _{1,2}^{\left( 1\right) }x_{1}x_{2}+\alpha _{2,1}^{\left(
1\right) }x_{2}x_{1}\right) =
\end{equation*}%
\begin{equation*}
\alpha _{1,2}^{\left( 1\right) }\underset{\left( 2\right) }{\ast }\left(
x_{1}\underset{\left( 2\right) }{\times }x_{2}\right) \underset{\left(
2\right) }{\perp }\alpha _{2,1}^{\left( 1\right) }\underset{\left( 2\right) }%
{\ast }\left( x_{2}\underset{\left( 2\right) }{\times }x_{1}\right) =
\end{equation*}%
\begin{equation*}
\varphi ^{\left( 2\right) }\left( \alpha _{1,2}^{\left( 1\right) }\right)
\left( \alpha _{1,2}^{\left( 2\right) }x_{1}x_{2}+\alpha _{2,1}^{\left(
2\right) }x_{2}x_{1}\right) +\varphi ^{\left( 2\right) }\left( \alpha
_{2,1}^{\left( 1\right) }\right) \left( \alpha _{1,2}^{\left( 2\right)
}x_{2}x_{1}+\alpha _{2,1}^{\left( 2\right) }x_{1}x_{2}\right) =
\end{equation*}%
\begin{equation*}
\left( \varphi ^{\left( 2\right) }\left( \alpha _{1,2}^{\left( 1\right)
}\right) \alpha _{1,2}^{\left( 2\right) }+\varphi ^{\left( 2\right) }\left(
\alpha _{2,1}^{\left( 1\right) }\right) \alpha _{2,1}^{\left( 2\right)
}\right) x_{1}x_{2}+\left( \varphi ^{\left( 2\right) }\left( \alpha
_{1,2}^{\left( 1\right) }\right) \alpha _{2,1}^{\left( 2\right) }+\varphi
^{\left( 2\right) }\left( \alpha _{2,1}^{\left( 1\right) }\right) \alpha
_{1,2}^{\left( 2\right) }\right) x_{2}x_{1}.
\end{equation*}%
Therefore%
\begin{equation*}
\sigma \left( \varphi ^{\left( 2\right) },\gamma _{1,2}^{\left( 2\right)
},\alpha _{1,2}^{\left( 2\right) },\alpha _{2,1}^{\left( 2\right) }\right)
\sigma \left( \varphi ^{\left( 1\right) },\gamma _{1,2}^{\left( 1\right)
},\alpha _{1,2}^{\left( 1\right) },\alpha _{2,1}^{\left( 1\right) }\right)
=\sigma \left( \varphi ^{\left( 3\right) },\gamma _{1,2}^{\left( 3\right)
},\alpha _{1,2}^{\left( 3\right) },\alpha _{2,1}^{\left( 3\right) }\right) ,
\end{equation*}%
where%
\begin{equation*}
\varphi ^{\left( 3\right) }=\varphi ^{\left( 2\right) }\varphi ^{\left(
1\right) },\gamma _{1,2}^{\left( 3\right) }=\gamma _{1,2}^{\left( 2\right)
}+\varphi ^{\left( 2\right) }\left( \gamma _{1,2}^{\left( 1\right) }\right)
\left( \alpha _{1,2}^{\left( 2\right) }+\alpha _{2,1}^{\left( 2\right)
}\right) 
\end{equation*}%
\begin{equation}
\alpha _{1,2}^{\left( 3\right) }=\varphi ^{\left( 2\right) }\left( \alpha
_{1,2}^{\left( 1\right) }\right) \alpha _{1,2}^{\left( 2\right) }+\varphi
^{\left( 2\right) }\left( \alpha _{2,1}^{\left( 1\right) }\right) \alpha
_{2,1}^{\left( 2\right) },  \label{homom_mult_3}
\end{equation}%
\begin{equation*}
\alpha _{2,1}^{\left( 3\right) }=\varphi ^{\left( 2\right) }\left( \alpha
_{1,2}^{\left( 1\right) }\right) \alpha _{2,1}^{\left( 2\right) }+\varphi
^{\left( 2\right) }\left( \alpha _{2,1}^{\left( 1\right) }\right) \alpha
_{1,2}^{\left( 2\right) }.
\end{equation*}%
Hence we have a decomposition%
\begin{equation}
\sigma \left( \varphi ,\gamma _{1,2},\alpha _{1,2},\alpha _{2,1}\right)
=\sigma \left( id_{k},\gamma _{1,2},1,0\right) \sigma \left( \varphi
,0,\alpha _{1,2},\alpha _{2,1}\right) .  \label{autom_decomp_3}
\end{equation}%
Also we have for every $F\in \mathrm{Ob}\Theta _{3}^{0}$ by (\ref%
{homom_mult_3})%
\begin{equation*}
\sigma _{F}\left( id_{k},-\gamma _{1,2},1,0\right) =\sigma _{F}\left(
id_{k},\gamma _{1,2},1,0\right) ^{-1},
\end{equation*}%
\begin{equation*}
\sigma _{F}\left( \varphi ^{-1},0,\beta _{1,2},\beta _{2,1}\right) =\sigma
_{F}\left( \varphi ,0,\alpha _{1,2},\alpha _{2,1}\right) ^{-1},
\end{equation*}%
where 
\begin{equation*}
\left( 
\begin{array}{cc}
\beta _{1,2} & \beta _{2,1} \\ 
\beta _{2,1} & \beta _{1,2}%
\end{array}%
\right) =\left( 
\begin{array}{cc}
\varphi ^{-1}\left( \alpha _{1,2}\right)  & \varphi ^{-1}\left( \alpha
_{2,1}\right)  \\ 
\varphi ^{-1}\left( \alpha _{2,1}\right)  & \varphi ^{-1}\left( \alpha
_{1,2}\right) 
\end{array}%
\right) ^{-1}.
\end{equation*}%
Therefore all homomorphisms $\sigma _{F}\left( \varphi ,\gamma _{1,2},\alpha
_{1,2},\alpha _{2,1}\right) $ for every $F\in \mathrm{Ob}\Theta _{3}^{0}$,
every $\varphi \in \mathrm{Aut}k$ and $\gamma _{1,2},\alpha _{1,2},\alpha
_{2,1}\in k$ such that $\alpha _{1,2}\neq \pm \alpha _{2,1}$ are
isomorphisms. So all systems of words $W$ which have form (\ref%
{real_list_nilp3}) fulfill conditions Op1) and Op2) and all automorphisms of 
$\Theta _{3}^{0}$ which corresponds to these systems of words are strongly
stable.

Now we must calculate the group $\mathfrak{Y}_{3}\cap \mathfrak{S}_{3}$. If $%
\Phi \in \mathfrak{S}_{3}$ then the system of words $W$ which corresponds to 
$\Phi $ fulfills conditions Op1) and Op2). From Proposition \ref{F*i=Fi} we
can conclude that $\sigma _{F}\left( F^{i}\right) =F^{i}$ for every $F\in 
\mathrm{Ob}\Theta _{3}^{0}$ and $i=1,2$. Therefore as in the \cite[Lemma 4.1
and Proposition 4.1]{Tsurkov} we can prove that if $\Phi $ is inner then $%
\alpha _{2,1}=0$, $\varphi =id_{k}$. Now we will prove that this conditions
are sufficient. We consider $\Phi \in \mathfrak{S}_{3}$ which corresponds to
the system of words $W$ which has form (\ref{real_list_nilp3}) such that $%
\alpha _{2,1}=0$ and $\varphi =id_{k}$. $\alpha _{1,2}\neq \pm \alpha _{2,1}$%
, so $\alpha _{1,2}\neq 0$. For every $F\in \mathrm{Ob}\Theta _{3}^{0}$ we
define the mapping $c_{F}:F\rightarrow F$ by this formula:%
\begin{equation*}
c_{F}\left( f\right) =\alpha _{1,2}^{-1}f+\alpha _{1,2}^{-2}\gamma
_{1,2}f^{2},
\end{equation*}%
where $f\in F$. We will prove by direct calculation that $c_{F}:F\rightarrow
F_{W}^{\ast }$ is an isomorphism. For every $f\in F$ and $\lambda \in k$ we
have%
\begin{equation*}
c_{F}\left( \lambda f\right) =\alpha _{1,2}^{-1}\lambda f+\alpha
_{1,2}^{-2}\gamma _{1,2}\lambda ^{2}f^{2},
\end{equation*}%
\begin{equation*}
\lambda \ast c_{F}\left( f\right) =\lambda \left( \alpha _{1,2}^{-1}f+\alpha
_{1,2}^{-2}\gamma _{1,2}f^{2}\right) +\gamma _{1,2}\alpha _{1,2}^{-2}\left(
\lambda ^{2}-\lambda \right) f^{2}=c_{F}\left( \lambda f\right) .
\end{equation*}%
For every $f_{1},f_{2}\in F$ we have%
\begin{equation*}
c_{F}\left( f_{1}+f_{2}\right) =\alpha _{1,2}^{-1}\left( f_{1}+f_{2}\right)
+\alpha _{1,2}^{-2}\gamma _{1,2}\left( f_{1}+f_{2}\right) ^{2}=
\end{equation*}%
\begin{equation*}
\alpha _{1,2}^{-1}f_{1}+\alpha _{1,2}^{-1}f_{2}+\alpha _{1,2}^{-2}\gamma
_{1,2}f_{1}^{2}+\alpha _{1,2}^{-2}\gamma _{1,2}f_{1}f_{2}+\alpha
_{1,2}^{-2}\gamma _{1,2}f_{2}f_{1}+\alpha _{1,2}^{-2}\gamma _{1,2}f_{2}^{2},
\end{equation*}%
\begin{equation*}
c_{F}\left( f_{1}\right) \perp c_{F}\left( f_{2}\right) =\left( \alpha
_{1,2}^{-1}f_{1}+\alpha _{1,2}^{-2}\gamma _{1,2}f_{1}^{2}\right) \perp
\left( \alpha _{1,2}^{-1}f_{2}+\alpha _{1,2}^{-2}\gamma
_{1,2}f_{2}^{2}\right) =
\end{equation*}%
\begin{equation*}
\alpha _{1,2}^{-1}f_{1}+\alpha _{1,2}^{-2}\gamma _{1,2}f_{1}^{2}+\alpha
_{1,2}^{-1}f_{2}+\alpha _{1,2}^{-2}\gamma _{1,2}f_{2}^{2}+\alpha
_{1,2}^{-2}\gamma _{1,2}f_{1}f_{2}+\alpha _{1,2}^{-2}\gamma _{1,2}f_{2}f_{1}%
\text{,}
\end{equation*}%
\begin{equation*}
c_{F}\left( f_{1}f_{2}\right) =\alpha _{1,2}^{-1}f_{1}f_{2},
\end{equation*}%
\begin{equation*}
c_{F}\left( f_{1}\right) \times c_{F}\left( f_{2}\right) =\left( \alpha
_{1,2}^{-1}f_{1}+\alpha _{1,2}^{-2}\gamma _{1,2}f_{1}^{2}\right) \times
\left( \alpha _{1,2}^{-1}f_{2}+\alpha _{1,2}^{-2}\gamma
_{1,2}f_{2}^{2}\right) =\alpha _{1,2}^{-1}f_{1}f_{2}.
\end{equation*}%
Therefore $c_{F}$ is a homomorphism. By Proposition \ref{F*i=Fi} we have $%
F^{i}=\left( F_{W}^{\ast }\right) ^{i}$ for $i=1,2$. By (\ref%
{real_list_nilp3}) $\lambda \ast \left( f+\left( F_{W}^{\ast }\right)
^{2}\right) =\lambda f+\left( F_{W}^{\ast }\right) ^{2}$, so we have for
every $F\left( X\right) \in \mathrm{Ob}\Theta _{3}^{0}$ that $\mathrm{sp}%
\left\{ c_{F}\left( x\right) +\left( F_{W}^{\ast }\right) ^{2}\mid x\in
X\right\} =F_{W}^{\ast }/\left( F_{W}^{\ast }\right) ^{2}$. By \cite[%
Proposition 4.1]{TsurkovAutomEquiv} $F_{W}^{\ast }\in \Theta _{3}$. So $%
\left\langle c_{F}\left( x\right) +\left( F_{W}^{\ast }\right) ^{2}\mid x\in
X\right\rangle _{alg}=F_{W}^{\ast }$ and $c_{F}$ is an epimorphism. By the
dimensional reason $c_{F}$ is an isomorphism. It is clear that $c_{F}\psi
=\psi c_{D}$ fulfills for every $F,D\in \mathrm{Ob}\Theta _{3}^{0}$ and
every $\left( \psi :D\rightarrow F\right) \in \mathrm{Mor}\Theta _{3}^{0}$.
Therefore $\Phi $ is inner.

After this by using of the (\ref{autom_decomp_3}) we prove as in the \cite%
{Tsurkov} that 
\begin{equation*}
\mathfrak{A}_{3}\mathfrak{/Y}_{3}\mathfrak{\cong }\left( U\left( k\mathbf{S}%
_{\mathbf{2}}\right) \mathfrak{/}U\left( k\left\{ e\right\} \right) \right) 
\mathfrak{\leftthreetimes }\mathrm{Aut}k\mathfrak{\cong }k^{\ast }\mathfrak{%
\leftthreetimes }\mathrm{Aut}k.
\end{equation*}
\end{proof}

\subsection{Group $\mathfrak{A}_{4}\mathfrak{/Y}_{4}$.}

\begin{theorem}
$\mathfrak{A_{4}/Y}_{4}\cong k^{\ast }\mathfrak{\leftthreetimes }\mathrm{Aut}%
k$.
\end{theorem}

\begin{proof}
We assume that the system of words $W$ fulfills conditions Op1) and Op2) in
the variety $\Theta _{4}$. We will find the specific form of these words. By
Corollary \ref{rez} from Proposition \ref{sshomom} we have $%
w_{+,4}=w_{+,3}+r_{+,4}$, where $w_{+,4}\in F^{\left( 4\right) }\left(
x_{1},x_{2}\right) $, $w_{+,3}=x_{1}+x_{2}+\gamma _{1,2}x_{1}x_{2}+\gamma
_{1,2}x_{2}x_{1}\in \bigoplus\limits_{i=1}^{2}\left( F^{\left( 4\right)
}\left( x_{1},x_{2}\right) \right) _{i}$, $r_{+,4}\in \left( F^{\left(
4\right) }\left( x_{1},x_{2}\right) \right) _{3}$. We denote $%
r_{+,4}=\sum\limits_{i_{1},i_{2},i_{3}=1}^{2}\gamma _{\left(
i_{1},i_{2}\right) i_{3}}\left( x_{i_{1}}x_{i_{2}}\right)
x_{i_{3}}+\sum\limits_{i_{1},i_{2},i_{3}=1}^{2}\gamma _{i_{1}\left(
i_{2},i_{3}\right) }x_{i_{1}}\left( x_{i_{2}}x_{i_{3}}\right) $, where $%
\gamma _{\left( i_{1},i_{2}\right) i_{3}},\gamma _{i_{1}\left(
i_{2},i_{3}\right) }\in k$. From $w_{+,4}\left( x_{1},0\right) =x_{1}$ and $%
w_{+,4}\left( 0,x_{2}\right) =x_{2}$ we conclude that 
\begin{equation}
r_{+,4}\left( x_{1},0\right) =r_{+,4}\left( 0,x_{2}\right) =0.
\label{add_0_concl_4}
\end{equation}%
From 
\begin{equation}
w_{+,4}\left( x_{1},x_{2}\right) =w_{+,4}\left( x_{2},x_{1}\right) 
\label{add_comm_4}
\end{equation}%
we conclude that $\gamma _{\left( 1,1\right) 2}=\gamma _{\left( 2,2\right) 1}
$ and so on. 
\begin{equation}
w_{+,4}\left( x_{1},w_{+,4}\left( x_{2},x_{3}\right) \right) =w_{+,4}\left(
w_{+,4}\left( x_{1},x_{2}\right) ,x_{3}\right)   \label{add_ass_4_1}
\end{equation}%
must fulfills in $F^{\left( 4\right) }\left( x_{1},x_{2},x_{3}\right) $. We
can write the left side of this equation as $w_{+,3}\left(
x_{1},w_{+,4}\left( x_{2},x_{3}\right) \right) +r_{+,4}\left(
x_{1},w_{+,4}\left( x_{2},x_{3}\right) \right) $. By nilpotence $%
r_{+,4}\left( x_{1},w_{+,4}\left( x_{2},x_{3}\right) \right) =r_{+,4}\left(
x_{1},x_{2}+x_{3}\right) $. The terms of degree $3$ of the $w_{+,3}\left(
x_{1},w_{+,4}\left( x_{2},x_{3}\right) \right) $ are equal 
\begin{equation*}
r_{+,4}\left( x_{2},x_{3}\right) +\gamma _{1,2}^{2}x_{1}\left(
x_{2}x_{3}+x_{3}x_{2}\right) +\gamma _{1,2}^{2}\left(
x_{2}x_{3}+x_{3}x_{2}\right) x_{1}.
\end{equation*}%
Reciprocally the terms of degree $3$ of the right side of equation (\ref%
{add_ass_4_1}) are equal%
\begin{equation*}
r_{+,4}\left( x_{1},x_{2}\right) +\gamma _{1,2}^{2}\left(
x_{1}x_{2}+x_{2}x_{1}\right) x_{3}+\gamma _{1,2}^{2}x_{3}\left(
x_{1}x_{2}+x_{2}x_{1}\right) +r_{+,4}\left( x_{1}+x_{2},x_{3}\right) .
\end{equation*}%
Therefore from (\ref{add_ass_4_1}) we can conclude%
\begin{equation*}
\gamma _{1,2}^{2}x_{1}\left( x_{2}x_{3}+x_{3}x_{2}\right) +\gamma
_{1,2}^{2}\left( x_{2}x_{3}+x_{3}x_{2}\right) x_{1}+
\end{equation*}%
\begin{equation*}
\left( r_{+,4}\left( x_{1},x_{2}+x_{3}\right) -r_{+,4}\left(
x_{1},x_{2}\right) -r_{+,4}\left( x_{1},x_{3}\right) \right) =
\end{equation*}%
\begin{equation}
\gamma _{1,2}^{2}\left( x_{1}x_{2}+x_{2}x_{1}\right) x_{3}+\gamma
_{1,2}^{2}x_{3}\left( x_{1}x_{2}+x_{2}x_{1}\right) +  \label{add_ass_4_2}
\end{equation}%
\begin{equation*}
\left( r_{+,4}\left( x_{1}+x_{2},x_{3}\right) -r_{+,4}\left(
x_{1},x_{3}\right) -r_{+,4}\left( x_{2},x_{3}\right) \right) .
\end{equation*}%
By (\ref{add_0_concl_4}) all terms of (\ref{add_ass_4_2}) have entries of $%
x_{1},x_{2}$ and $x_{3}$. If we compare the coefficients of $\left(
x_{1}x_{2}\right) x_{3}$ in (\ref{add_ass_4_2}), then we achieve 
\begin{equation}
\gamma _{\left( 1,2\right) 2}=\gamma _{1,2}^{2}+\gamma _{\left( 1,1\right)
2}.  \label{add_ass_4_3_1}
\end{equation}%
From comparison of the coefficients of $\left( x_{2}x_{1}\right) x_{3}$ we
achieve by using of (\ref{add_comm_4}) 
\begin{equation}
\gamma _{\left( 1,2\right) 1}=\gamma _{1,2}^{2}+\gamma _{\left( 1,1\right)
2}.  \label{add_ass_4_3_2}
\end{equation}%
By consideration of the coefficients of $\left( x_{1}x_{3}\right) x_{2}$ and
of the coefficients of $\left( x_{3}x_{1}\right) x_{2}$ we conclude $\gamma
_{\left( 1,2\right) 2}=\gamma _{\left( 1,2\right) 1}$, but this is a
corollary of (\ref{add_ass_4_3_1}) and (\ref{add_ass_4_3_2}). If we compare
the coefficients of $\left( x_{2}x_{3}\right) x_{1}$ we conclude the
equation (\ref{add_ass_4_3_2}) and if we compare the coefficients of $\left(
x_{3}x_{2}\right) x_{1}$ we conclude the equation (\ref{add_ass_4_3_1}).
From comparison of the coefficients of $x_{1}\left( x_{2}x_{3}\right) $ we
achieve 
\begin{equation}
\gamma _{1,2}^{2}+\gamma _{1\left( 2,2\right) }=\gamma _{1\left( 1,2\right)
}.  \label{add_ass_4_3_3}
\end{equation}%
By consideration of the coefficients of $x_{1}\left( x_{3}x_{2}\right) $ we
achieve 
\begin{equation}
\gamma _{1,2}^{2}+\gamma _{1\left( 2,2\right) }=\gamma _{1\left( 2,1\right)
}.  \label{add_ass_4_3_4}
\end{equation}%
If we consider the coefficients of $x_{2}\left( x_{1}x_{3}\right) $ or the
coefficients of $x_{2}\left( x_{3}x_{1}\right) $, then we have $\gamma
_{1\left( 2,1\right) }=\gamma _{1\left( 1,2\right) }$, but this is a
corollary of (\ref{add_ass_4_3_3}) and (\ref{add_ass_4_3_4}). If we compare
the coefficients of $x_{3}\left( x_{1}x_{2}\right) $, then we achieve the
equation (\ref{add_ass_4_3_4}), and if we compare the coefficients of $%
x_{3}\left( x_{2}x_{1}\right) $, then we achieve the equation (\ref%
{add_ass_4_3_3}). By (\ref{add_ass_4_3_1}) - (\ref{add_ass_4_3_4}) we have
that 
\begin{equation*}
r_{+,4}\left( x_{1},x_{2}\right) =\gamma _{\left( 1,1\right)
2}x_{1}^{2}x_{2}+\left( \gamma _{1,2}^{2}+\gamma _{\left( 1,1\right)
2}\right) \left( x_{1}x_{2}\right) x_{2}+\left( \gamma _{1,2}^{2}+\gamma
_{\left( 1,1\right) 2}\right) \left( x_{1}x_{2}\right) x_{1}+
\end{equation*}%
\begin{equation*}
\gamma _{\left( 1,1\right) 2}x_{2}^{2}x_{1}+\left( \gamma _{1,2}^{2}+\gamma
_{\left( 1,1\right) 2}\right) \left( x_{2}x_{1}\right) x_{1}+\left( \gamma
_{1,2}^{2}+\gamma _{\left( 1,1\right) 2}\right) \left( x_{2}x_{1}\right)
x_{2}+
\end{equation*}%
\begin{equation*}
\gamma _{1\left( 2,2\right) }x_{1}x_{2}^{2}+\left( \gamma _{1,2}^{2}+\gamma
_{1\left( 2,2\right) }\right) x_{1}\left( x_{1}x_{2}\right) +\left( \gamma
_{1,2}^{2}+\gamma _{1\left( 2,2\right) }\right) x_{1}\left(
x_{2}x_{1}\right) +
\end{equation*}%
\begin{equation}
\gamma _{1\left( 2,2\right) }x_{2}x_{1}^{2}+\left( \gamma _{1,2}^{2}+\gamma
_{1\left( 2,2\right) }\right) x_{2}\left( x_{2}x_{1}\right) +\left( \gamma
_{1,2}^{2}+\gamma _{1\left( 2,2\right) }\right) x_{2}\left(
x_{1}x_{2}\right) .  \label{add_4}
\end{equation}

By Corollary \ref{rez} from Proposition \ref{sshomom} we have $w_{\cdot
,4}=w_{\cdot ,3}+r_{\cdot ,4}$, where $w_{\cdot ,4}\in F^{\left( 4\right)
}\left( x_{1},x_{2}\right) $, $w_{\cdot ,3}=\alpha _{1,2}x_{1}x_{2}+\alpha
_{2,1}x_{2}x_{1}\in \bigoplus\limits_{i=1}^{2}\left( F^{\left( 4\right)
}\left( x_{1},x_{2}\right) \right) _{i}$, $r_{\cdot ,4}\in \left( F^{\left(
4\right) }\left( x_{1},x_{2}\right) \right) _{3}$. We denote $r_{\cdot
,4}=\sum\limits_{i_{1},i_{2},i_{3}=1}^{2}\alpha _{\left( i_{1},i_{2}\right)
i_{3}}\left( x_{i_{1}}x_{i_{2}}\right)
x_{i_{3}}+\sum\limits_{i_{1},i_{2},i_{3}=1}^{2}\alpha _{i_{1}\left(
i_{2},i_{3}\right) }x_{i_{1}}\left( x_{i_{2}}x_{i_{3}}\right) $, where $%
\alpha _{\left( i_{1},i_{2}\right) i_{3}},\alpha _{i_{1}\left(
i_{2},i_{3}\right) }\in k$. From $w_{\cdot ,4}\left( x_{1},0\right)
=w_{\cdot ,4}\left( 0,x_{2}\right) =0$ we conclude that 
\begin{equation}
r_{\cdot ,4}\left( x_{1},0\right) =r_{\cdot ,4}\left( 0,x_{2}\right) =0.
\label{mul_0_concl_4}
\end{equation}%
\begin{equation}
w_{\cdot ,4}\left( w_{+,4}\left( x_{1},x_{2}\right) ,x_{3}\right)
=w_{+,4}\left( w_{\cdot ,4}\left( x_{1},x_{3}\right) ,w_{\cdot ,4}\left(
x_{2},x_{3}\right) \right)   \label{distr_left_4_1}
\end{equation}%
must fulfills in $F^{\left( 4\right) }\left( x_{1},x_{2},x_{3}\right) $ by
Op2). By nilpotence 
\begin{equation*}
w_{+,4}\left( w_{\cdot ,4}\left( x_{1},x_{3}\right) ,w_{\cdot ,4}\left(
x_{2},x_{3}\right) \right) =w_{\cdot ,4}\left( x_{1},x_{3}\right) +w_{\cdot
,4}\left( x_{2},x_{3}\right) .
\end{equation*}%
If we compare the terms of degree $3$ of the equation (\ref{distr_left_4_1}%
), then we have%
\begin{equation*}
w_{\cdot ,3}\left( \gamma _{1,2}x_{1}x_{2}+\gamma
_{1,2}x_{2}x_{1},x_{3}\right) +r_{\cdot ,4}\left( x_{1}+x_{2},x_{3}\right)
=r_{\cdot ,4}\left( x_{1},x_{3}\right) +r_{\cdot ,4}\left(
x_{2},x_{3}\right) 
\end{equation*}%
or%
\begin{equation*}
\alpha _{1,2}\gamma _{1,2}\left( x_{1}x_{2}\right) x_{3}+\alpha _{1,2}\gamma
_{1,2}\left( x_{2}x_{1}\right) x_{3}+\alpha _{2,1}\gamma _{1,2}x_{3}\left(
x_{1}x_{2}\right) +\alpha _{2,1}\gamma _{1,2}x_{3}\left( x_{2}x_{1}\right) +
\end{equation*}%
\begin{equation}
\left( r_{\cdot ,4}\left( x_{1}+x_{2},x_{3}\right) -r_{\cdot ,4}\left(
x_{1},x_{3}\right) -r_{\cdot ,4}\left( x_{2},x_{3}\right) \right) =0.
\label{distr_left_4_2}
\end{equation}%
By (\ref{add_0_concl_4}) all terms of (\ref{add_ass_4_2}) have entries of $%
x_{1},x_{2}$ and $x_{3}$. If we compare the coefficients of $\left(
x_{1}x_{2}\right) x_{3}$ we achieve 
\begin{equation}
\alpha _{1,2}\gamma _{1,2}+\alpha _{\left( 1,1\right) 2}=0.
\label{distr_left_4_3_1}
\end{equation}%
Other comparisons in the order, which was described above, give us 
\begin{equation}
\alpha _{\left( 1,2\right) 1}=0,  \label{distr_left_4_3_2}
\end{equation}%
\begin{equation}
\alpha _{\left( 2,1\right) 1}=0,  \label{distr_left_4_3_3}
\end{equation}%
\begin{equation}
\alpha _{1\left( 1,2\right) }=0,  \label{distr_left_4_3_4}
\end{equation}%
\begin{equation}
\alpha _{1\left( 2,1\right) }=0,  \label{distr_left_4_3_5}
\end{equation}%
\begin{equation}
\alpha _{2,1}\gamma _{1,2}+\alpha _{2\left( 1,1\right) }=0.
\label{distr_left_4_3_6}
\end{equation}%
Here we omit the repeated equations.

Also%
\begin{equation}
w_{\cdot ,4}\left( x_{1},w_{+,4}\left( x_{2},x_{3}\right) \right)
=w_{+,4}\left( w_{\cdot ,4}\left( x_{1},x_{2}\right) ,w_{\cdot ,4}\left(
x_{1},x_{3}\right) \right)   \label{distr_right_4_1}
\end{equation}%
must fulfills in $F_{3}^{\left( 4\right) }$. As above we conclude from this
equation%
\begin{equation*}
\alpha _{1,2}\gamma _{1,2}x_{1}\left( x_{2}x_{3}\right) +\alpha _{1,2}\gamma
_{1,2}x_{1}\left( x_{3}x_{2}\right) +\alpha _{2,1}\gamma _{1,2}\left(
x_{2}x_{3}\right) x_{1}+\alpha _{2,1}\gamma _{1,2}\left( x_{3}x_{2}\right)
x_{1}=
\end{equation*}%
\begin{equation}
\left( r_{\cdot ,4}\left( x_{1},x_{2}+x_{3}\right) -r_{\cdot ,4}\left(
x_{1},x_{2}\right) -r_{\cdot ,4}\left( x_{1},x_{3}\right) \right) =0.
\label{distr_right_4_2}
\end{equation}%
After this, we have by comparisons of the coefficients of suitable monomials%
\begin{equation}
\alpha _{\left( 1,2\right) 2}=0,  \label{distr_right_4_3_1}
\end{equation}%
\begin{equation}
\alpha _{\left( 2,1\right) 2}=0,  \label{distr_right_4_3_2}
\end{equation}%
\begin{equation}
\alpha _{2,1}\gamma _{1,2}+\alpha _{\left( 2,2\right) 1}=0,
\label{distr_right_4_3_3}
\end{equation}%
\begin{equation}
\alpha _{1,2}\gamma _{1,2}+\alpha _{1\left( 2,2\right) }=0,
\label{distr_right_4_3_4}
\end{equation}%
\begin{equation}
\alpha _{2\left( 1,2\right) }=0,  \label{distr_right_4_3_5}
\end{equation}%
\begin{equation}
\alpha _{2\left( 2,1\right) }=0.  \label{distr_right_4_3_6}
\end{equation}%
Therefore%
\begin{equation}
r_{\cdot ,4}\left( x_{1},x_{2}\right) =-\alpha _{1,2}\gamma
_{1,2}x_{1}^{2}x_{2}-\alpha _{2,1}\gamma _{1,2}x_{2}x_{1}^{2}-\alpha
_{2,1}\gamma _{1,2}x_{2}^{2}x_{1}-\alpha _{1,2}\gamma _{1,2}x_{1}x_{2}^{2}.
\label{mul_4}
\end{equation}

Also by Corollary \ref{rez} from Proposition \ref{sshomom} we have $%
w_{\lambda ,4}\left( x\right) =w_{\lambda ,3}\left( x\right) +r_{\lambda
,4}\left( x\right) $, where $w_{\lambda ,3}\left( x\right) =\varphi \left(
\lambda \right) x+\gamma _{1,2}\left( \left( \varphi \left( \lambda \right)
\right) ^{2}-\varphi \left( \lambda \right) \right) x^{2}$, $r_{\lambda
,4}\left( x\right) =\psi _{1}\left( \lambda \right) x\left( x^{2}\right)
+\psi _{2}\left( \lambda \right) \left( x^{2}\right) x$. By Op2)%
\begin{equation}
w_{\lambda ,4}\left( w_{+,4}\left( x_{1},x_{2}\right) \right) =w_{+,4}\left(
w_{\lambda ,4}\left( x_{1}\right) ,w_{\lambda ,4}\left( x_{2}\right) \right) 
\label{scalar_distr_4_1}
\end{equation}%
must fulfills in $F^{\left( 4\right) }\left( x_{1},x_{2}\right) $. If we
compare the terms of degree $3$ of this equation, we achieve 
\begin{equation*}
\left( r_{\lambda ,4}\left( x_{1}+x_{2}\right) -r_{\lambda ,4}\left(
x_{1}\right) -r_{\lambda ,4}\left( x_{2}\right) \right) +
\end{equation*}%
\begin{equation*}
\gamma _{1,2}^{2}\left( \left( \varphi \left( \lambda \right) \right)
^{2}-\varphi \left( \lambda \right) \right) (x_{1}\left( x_{1}x_{2}\right)
+x_{1}\left( x_{2}x_{1}\right) +x_{2}\left( x_{1}x_{2}\right) +x_{2}\left(
x_{2}x_{1}\right) +
\end{equation*}%
\begin{equation}
+\left( x_{1}x_{2}\right) x_{1}+\left( x_{1}x_{2}\right) x_{2}+\left(
x_{2}x_{1}\right) x_{1}+\left( x_{2}x_{1}\right) x_{2})=
\label{scalar_distr_4_2}
\end{equation}%
\begin{equation*}
\left( \varphi \left( \lambda \right) ^{3}-\varphi \left( \lambda \right)
\right) r_{+,4}\left( x_{1},x_{2}\right) +
\end{equation*}%
\begin{equation*}
\gamma _{1,2}^{2}\varphi \left( \lambda \right) \left( \left( \varphi \left(
\lambda \right) \right) ^{2}-\varphi \left( \lambda \right) \right) \left(
x_{1}x_{2}^{2}+x_{1}^{2}x_{2}+x_{2}x_{1}^{2}+x_{2}^{2}x_{1}\right) .
\end{equation*}%
It is clear that in $r_{\lambda ,4}\left( x_{1}+x_{2}\right) -r_{\lambda
,4}\left( x_{1}\right) -r_{\lambda ,4}\left( x_{2}\right) $ there are no
terms in which there are only entries of $x_{1}$ or only entries of $x_{2}$.
If we compare coefficients of $x_{1}^{2}x_{2}$ in the equation (\ref%
{scalar_distr_4_2}) we achieve%
\begin{equation}
\psi _{2}\left( \lambda \right) =\gamma _{1,2}^{2}\varphi \left( \lambda
\right) \left( \left( \varphi \left( \lambda \right) \right) ^{2}-\varphi
\left( \lambda \right) \right) +\gamma _{\left( 1,1\right) 2}\left( \varphi
\left( \lambda \right) ^{3}-\varphi \left( \lambda \right) \right) .
\label{scalar_distr_4_3_1}
\end{equation}%
if we compare coefficients of $x_{1}x_{2}^{2}$ in the equation (\ref%
{scalar_distr_4_2}) we achieve%
\begin{equation}
\psi _{1}\left( \lambda \right) =\gamma _{1,2}^{2}\varphi \left( \lambda
\right) \left( \left( \varphi \left( \lambda \right) \right) ^{2}-\varphi
\left( \lambda \right) \right) +\gamma _{1\left( 2,2\right) }\left( \varphi
\left( \lambda \right) ^{3}-\varphi \left( \lambda \right) \right) .
\label{scalar_distr_4_3_2}
\end{equation}%
Other comparisons give us only repetitions of the equations (\ref%
{scalar_distr_4_3_1}) and (\ref{scalar_distr_4_3_2}). So%
\begin{equation*}
r_{\lambda ,4}\left( x\right) =\left( \gamma _{1,2}^{2}\left( \varphi \left(
\lambda \right) ^{3}-\varphi \left( \lambda \right) ^{2}\right) +\gamma
_{1\left( 2,2\right) }\left( \varphi \left( \lambda \right) ^{3}-\varphi
\left( \lambda \right) \right) \right) x\left( x^{2}\right) +
\end{equation*}%
\begin{equation}
\left( \gamma _{1,2}^{2}\left( \varphi \left( \lambda \right) ^{3}-\varphi
\left( \lambda \right) ^{2}\right) +\gamma _{\left( 1,1\right) 2}\left(
\varphi \left( \lambda \right) ^{3}-\varphi \left( \lambda \right) \right)
\right) \left( x^{2}\right) x.  \label{mul_scal_4}
\end{equation}

Now we will prove that if the system of words $W$ defined by formulas (\ref%
{real_list_nilp3}), (\ref{add_4}), (\ref{mul_scal_4}) and (\ref{mul_4}) then
this system of words fulfills conditions Op1) and Op2). First of all we must
check that if $H\in \Theta _{4}$ then $H_{W}^{\ast }\in \Theta _{4}$. We can
check it by direct calculation. For example, if $h\in H\in \Theta _{3}$, $%
\mu ,\lambda \in k$ then%
\begin{equation*}
\left( \mu \lambda \right) \ast h=\varphi \left( \mu \lambda \right)
h+\gamma _{1,2}\left( \varphi \left( \mu \lambda \right) ^{2}-\varphi \left(
\mu \lambda \right) \right) h^{2}+
\end{equation*}%
\begin{equation*}
\left( \gamma _{1,2}^{2}\left( \varphi \left( \mu \lambda \right)
^{3}-\varphi \left( \mu \lambda \right) ^{2}\right) +\gamma _{1\left(
2,2\right) }\left( \varphi \left( \mu \lambda \right) ^{3}-\varphi \left(
\mu \lambda \right) \right) \right) h\left( h^{2}\right) +
\end{equation*}%
\begin{equation*}
\left( \gamma _{1,2}^{2}\left( \varphi \left( \mu \lambda \right)
^{3}-\varphi \left( \mu \lambda \right) ^{2}\right) +\gamma _{\left(
1,1\right) 2}\left( \varphi \left( \mu \lambda \right) ^{3}-\varphi \left(
\mu \lambda \right) \right) \right) \left( h^{2}\right) h.
\end{equation*}%
\begin{equation*}
\mu \ast \left( \lambda \ast h\right) =
\end{equation*}%
\begin{equation*}
\mu \ast (\varphi \left( \lambda \right) h+\gamma _{1,2}\left( \varphi
\left( \lambda \right) ^{2}-\varphi \left( \lambda \right) \right) h^{2}+
\end{equation*}%
\begin{equation*}
\left( \gamma _{1,2}^{2}\left( \varphi \left( \lambda \right) ^{3}-\varphi
\left( \lambda \right) ^{2}\right) +\gamma _{1\left( 2,2\right) }\left(
\varphi \left( \lambda \right) ^{3}-\varphi \left( \lambda \right) \right)
\right) h\left( h^{2}\right) +
\end{equation*}%
\begin{equation*}
\left( \gamma _{1,2}^{2}\left( \varphi \left( \lambda \right) ^{3}-\varphi
\left( \lambda \right) ^{2}\right) +\gamma _{\left( 1,1\right) 2}\left(
\varphi \left( \lambda \right) ^{3}-\varphi \left( \lambda \right) \right)
\right) \left( h^{2}\right) h=
\end{equation*}%
\begin{equation*}
\varphi \left( \mu \right) \varphi \left( \lambda \right) h+\gamma
_{1,2}\varphi \left( \mu \right) \left( \varphi \left( \lambda \right)
^{2}-\varphi \left( \lambda \right) \right) h^{2}+
\end{equation*}%
\begin{equation*}
\left( \gamma _{1,2}^{2}\varphi \left( \mu \right) \left( \varphi \left(
\lambda \right) ^{3}-\varphi \left( \lambda \right) ^{2}\right) +\gamma
_{1\left( 2,2\right) }\varphi \left( \mu \right) \left( \varphi \left(
\lambda \right) ^{3}-\varphi \left( \lambda \right) \right) \right) h\left(
h^{2}\right) +
\end{equation*}%
\begin{equation*}
\left( \gamma _{1,2}^{2}\varphi \left( \mu \right) \left( \varphi \left(
\lambda \right) ^{3}-\varphi \left( \lambda \right) ^{2}\right) +\gamma
_{\left( 1,1\right) 2}\varphi \left( \mu \right) \left( \varphi \left(
\lambda \right) ^{3}-\varphi \left( \lambda \right) \right) \right) \left(
h^{2}\right) h+
\end{equation*}%
\begin{equation*}
\gamma _{1,2}\left( \varphi \left( \mu \right) ^{2}-\varphi \left( \mu
\right) \right) (\varphi \left( \lambda \right) ^{2}h^{2}+
\end{equation*}%
\begin{equation*}
\gamma _{1,2}\left( \varphi \left( \lambda \right) ^{3}-\varphi \left(
\lambda \right) ^{2}\right) h\left( h^{2}\right) +\gamma _{1,2}\left(
\varphi \left( \lambda \right) ^{3}-\varphi \left( \lambda \right)
^{2}\right) \left( h^{2}\right) h)+
\end{equation*}%
\begin{equation*}
\left( \gamma _{1,2}^{2}\left( \varphi \left( \mu \right) ^{3}-\varphi
\left( \mu \right) ^{2}\right) +\gamma _{1\left( 2,2\right) }\left( \varphi
\left( \mu \right) ^{3}-\varphi \left( \mu \right) \right) \right) \varphi
\left( \lambda \right) ^{3}h\left( h^{2}\right) +
\end{equation*}%
\begin{equation*}
\left( \gamma _{1,2}^{2}\left( \varphi \left( \mu \right) ^{3}-\varphi
\left( \mu \right) ^{2}\right) +\gamma _{\left( 1,1\right) 2}\left( \varphi
\left( \mu \right) ^{3}-\varphi \left( \mu \right) \right) \right) \varphi
\left( \lambda \right) ^{3}\left( h^{2}\right) h=
\end{equation*}%
\begin{equation*}
\varphi \left( \mu \lambda \right) h+\gamma _{1,2}\left( \varphi \left( \mu
\right) \left( \varphi \left( \lambda \right) ^{2}-\varphi \left( \lambda
\right) \right) +\left( \varphi \left( \mu \right) ^{2}-\varphi \left( \mu
\right) \right) \varphi \left( \lambda \right) ^{2}\right) h^{2}+
\end{equation*}%
\begin{equation*}
(\gamma _{1,2}^{2}(\varphi \left( \mu \right) \left( \varphi \left( \lambda
\right) ^{3}-\varphi \left( \lambda \right) ^{2}\right) +\left( \varphi
\left( \mu \right) ^{2}-\varphi \left( \mu \right) \right) \left( \varphi
\left( \lambda \right) ^{3}-\varphi \left( \lambda \right) ^{2}\right) +
\end{equation*}%
\begin{equation*}
\left( \varphi \left( \mu \right) ^{3}-\varphi \left( \mu \right)
^{2}\right) \varphi \left( \lambda \right) ^{3})+
\end{equation*}%
\begin{equation*}
\gamma _{1\left( 2,2\right) }\left( \varphi \left( \mu \right) \left(
\varphi \left( \lambda \right) ^{3}-\varphi \left( \lambda \right) \right)
+\left( \varphi \left( \mu \right) ^{3}-\varphi \left( \mu \right) \right)
\varphi \left( \lambda \right) ^{3}\right) )h\left( h^{2}\right) +
\end{equation*}%
\begin{equation*}
(\gamma _{1,2}^{2}(\varphi \left( \mu \right) \left( \varphi \left( \lambda
\right) ^{3}-\varphi \left( \lambda \right) ^{2}\right) +\left( \varphi
\left( \mu \right) ^{2}-\varphi \left( \mu \right) \right) \left( \varphi
\left( \lambda \right) ^{3}-\varphi \left( \lambda \right) ^{2}\right) +
\end{equation*}%
\begin{equation*}
\left( \varphi \left( \mu \right) ^{3}-\varphi \left( \mu \right)
^{2}\right) \varphi \left( \lambda \right) ^{3})+
\end{equation*}%
\begin{equation*}
\gamma _{\left( 1,1\right) 2}\left( \varphi \left( \mu \right) \left(
\varphi \left( \lambda \right) ^{3}-\varphi \left( \lambda \right) \right)
+\left( \varphi \left( \mu \right) ^{3}-\varphi \left( \mu \right) \right)
\varphi \left( \lambda \right) ^{3}\right) )\left( h^{2}\right) h=
\end{equation*}%
\begin{equation*}
\varphi \left( \mu \lambda \right) h+\gamma _{1,2}\left( \varphi \left( \mu
\lambda \right) ^{2}-\varphi \left( \mu \lambda \right) \right) h^{2}+
\end{equation*}%
\begin{equation*}
\left( \gamma _{1,2}^{2}\left( \varphi \left( \mu \lambda \right)
^{3}-\varphi \left( \mu \lambda \right) ^{2}\right) +\gamma _{1\left(
2,2\right) }\left( \varphi \left( \mu \lambda \right) ^{3}-\varphi \left(
\mu \lambda \right) \right) \right) h\left( h^{2}\right) +
\end{equation*}%
\begin{equation*}
\left( \gamma _{1,2}^{2}\left( \varphi \left( \mu \lambda \right)
^{3}-\varphi \left( \mu \lambda \right) ^{2}\right) +\gamma _{\left(
1,1\right) 2}\left( \varphi \left( \mu \lambda \right) ^{3}-\varphi \left(
\mu \lambda \right) \right) \right) \left( h^{2}\right) h.
\end{equation*}%
Also by direct calculation we can check that for every $h_{1},h_{2}\in H\in
\Theta _{3}$, and every $\lambda \in k$ the%
\begin{equation*}
\lambda \ast \left( h_{1}\times h_{2}\right) =\left( \lambda \ast
h_{1}\right) \times h_{2}=h_{1}\times \left( \lambda \ast h_{2}\right) 
\end{equation*}%
holds. Other axioms of the variety $\Theta _{4}$ fulfill in $H_{W}^{\ast }$
by the constructions of the words of the system $W$.

So we can conclude that for every $F=F\left( X\right) \in \mathrm{Ob}\Theta
_{4}^{0}$ exists a homomorphism $\sigma _{F}:F\rightarrow F_{W}^{\ast }$
such that $\sigma _{F}\mid _{X}=id_{X}$. Our goal is to prove that these
homomorphisms are isomorphisms. For this purpose we will research the
superpositions of these homomorphisms. We will not lead out the general
formula of the superposition, similar to the formula (\ref{homom_mult_3}).
Hear we will consider only special cases of superpositions. If the system of
words $W$ defined by formulas (\ref{real_list_nilp3}), (\ref{add_4}), (\ref%
{mul_scal_4}) and (\ref{mul_4}) then its words and homomorphisms $\sigma
_{F}:F\rightarrow F_{W}^{\ast }$ depend on parameters $\varphi \in \mathrm{%
Aut}k$ and $\gamma _{1,2},\gamma _{1\left( 2,2\right) },\gamma _{\left(
1,1\right) 2},\alpha _{1,2},\alpha _{2,1}\in k$ such that $\alpha _{1,2}\neq
\pm \alpha _{2,1}$. So we denote 
\begin{equation*}
\sigma _{F}=\sigma _{F}\left( \varphi ,\gamma _{1,2},\gamma _{1\left(
2,2\right) },\gamma _{\left( 1,1\right) 2},\alpha _{1,2},\alpha
_{2,1}\right) 
\end{equation*}%
or 
\begin{equation*}
\sigma _{F}=\sigma \left( \varphi ,\gamma _{1,2},\gamma _{1\left( 2,2\right)
},\gamma _{\left( 1,1\right) 2},\alpha _{1,2},\alpha _{2,1}\right) .
\end{equation*}%
Sometimes we will give the indexes to these parameters. We have that%
\begin{equation*}
\sigma \left( id_{k},\gamma _{1,2},\gamma _{1\left( 2,2\right) },\gamma
_{\left( 1,1\right) 2},1,0\right) \sigma \left( \varphi ,0,0,0,\alpha
_{1,2},\alpha _{2,1}\right) \left( x_{1}+x_{2}\right) =
\end{equation*}%
\begin{equation*}
\sigma \left( id_{k},\gamma _{1,2},\gamma _{1\left( 2,2\right) },\gamma
_{\left( 1,1\right) 2},1,0\right) \left( x_{1}+x_{2}\right) =
\end{equation*}%
\begin{equation*}
x_{1}+x_{2}+\gamma _{1,2}x_{1}x_{2}+\gamma _{1,2}x_{2}x_{1}+
\end{equation*}%
\begin{equation*}
\left( \gamma _{1,2}^{2}+\gamma _{1\left( 2,2\right) }\right) x_{1}\left(
x_{1}x_{2}\right) +\left( \gamma _{1,2}^{2}+\gamma _{1\left( 2,2\right)
}\right) x_{1}\left( x_{2}x_{1}\right) +
\end{equation*}%
\begin{equation*}
\gamma _{1\left( 2,2\right) }x_{1}\left( x_{2}^{2}\right) +\gamma _{1\left(
2,2\right) }x_{2}\left( x_{1}^{2}\right) +
\end{equation*}%
\begin{equation*}
\left( \gamma _{1,2}^{2}+\gamma _{1\left( 2,2\right) }\right) x_{2}\left(
x_{1}x_{2}\right) +\left( \gamma _{1,2}^{2}+\gamma _{1\left( 2,2\right)
}\right) x_{2}\left( x_{2}x_{1}\right) +
\end{equation*}%
\begin{equation*}
\left( \gamma _{1,2}^{2}+\gamma _{\left( 1,1\right) 2}\right) \left(
x_{1}x_{2}\right) x_{1}+\left( \gamma _{1,2}^{2}+\gamma _{\left( 1,1\right)
2}\right) \left( x_{2}x_{1}\right) x_{1}+
\end{equation*}%
\begin{equation*}
\gamma _{\left( 1,1\right) 2}\left( x_{2}^{2}\right) x_{1}+\gamma _{\left(
1,1\right) 2}\left( x_{1}^{2}\right) x_{2}+
\end{equation*}%
\begin{equation*}
\left( \gamma _{1,2}^{2}+\gamma _{\left( 1,1\right) 2}\right) \left(
x_{1}x_{2}\right) x_{2}+\left( \gamma _{1,2}^{2}+\gamma _{\left( 1,1\right)
2}\right) \left( x_{2}x_{1}\right) x_{2}.
\end{equation*}%
\begin{equation*}
\sigma \left( id_{k},\gamma _{1,2},\gamma _{1\left( 2,2\right) },\gamma
_{\left( 1,1\right) 2},1,0\right) \sigma \left( \varphi ,0,0,0,\alpha
_{1,2},\alpha _{2,1}\right) \left( x_{1}x_{2}\right) =
\end{equation*}%
\begin{equation*}
\sigma \left( id_{k},\gamma _{1,2},\gamma _{1\left( 2,2\right) },\gamma
_{\left( 1,1\right) 2},1,0\right) \left( \alpha _{1,2}x_{1}x_{2}+\alpha
_{2,1}x_{2}x_{1}\right) =
\end{equation*}%
\begin{equation*}
\alpha _{1,2}\ast \left( x_{1}\times x_{2}\right) \perp \alpha _{2,1}\ast
\left( x_{2}\times x_{1}\right) ,
\end{equation*}%
where operations $\perp $, $\times $ and $\ast $ defined by the system of
words which depend on parameters $id_{k}$, $\gamma _{1,2}$, $\gamma
_{1\left( 2,2\right) }$, $\gamma _{\left( 1,1\right) 2}$, $1$, $0$. So%
\begin{equation*}
\alpha _{1,2}\ast \left( x_{1}\times x_{2}\right) =\alpha
_{1,2}x_{1}x_{2}-\alpha _{1,2}\gamma _{1,2}x_{1}\left( x_{2}^{2}\right)
-\alpha _{1,2}\gamma _{1,2}\left( x_{1}^{2}\right) x_{2}
\end{equation*}%
and%
\begin{equation*}
\sigma \left( id_{k},\gamma _{1,2},\gamma _{1\left( 2,2\right) },\gamma
_{\left( 1,1\right) 2},1,0\right) \sigma \left( \varphi ,0,0,0,\alpha
_{1,2},\alpha _{2,1}\right) \left( x_{1}x_{2}\right) =
\end{equation*}%
\begin{equation*}
\alpha _{1,2}x_{1}x_{2}-\alpha _{1,2}\gamma _{1,2}x_{1}\left(
x_{2}^{2}\right) -\alpha _{1,2}\gamma _{1,2}\left( x_{1}^{2}\right) x_{2}+
\end{equation*}%
\begin{equation*}
\alpha _{2,1}x_{2}x_{1}-\alpha _{2,1}\gamma _{1,2}x_{2}\left(
x_{1}^{2}\right) -\alpha _{2,1}\gamma _{1,2}\left( x_{2}^{2}\right) x_{1}.
\end{equation*}%
\begin{equation*}
\sigma \left( id_{k},\gamma _{1,2},\gamma _{1\left( 2,2\right) },\gamma
_{\left( 1,1\right) 2},1,0\right) \sigma \left( \varphi ,0,0,0,\alpha
_{1,2},\alpha _{2,1}\right) \left( \lambda x\right) =
\end{equation*}%
\begin{equation*}
\sigma \left( id_{k},\gamma _{1,2},\gamma _{1\left( 2,2\right) },\gamma
_{\left( 1,1\right) 2},1,0\right) \left( \varphi \left( \lambda \right)
x\right) =
\end{equation*}%
\begin{equation*}
\varphi \left( \lambda \right) x+\gamma _{1,2}\left( \left( \varphi \left(
\lambda \right) \right) ^{2}-\varphi \left( \lambda \right) \right) x^{2}+
\end{equation*}%
\begin{equation*}
\left( \gamma _{1,2}^{2}\left( \varphi \left( \lambda \right) ^{3}-\varphi
\left( \lambda \right) ^{2}\right) +\gamma _{1\left( 2,2\right) }\left(
\varphi \left( \lambda \right) ^{3}-\varphi \left( \lambda \right) \right)
\right) x\left( x^{2}\right) +
\end{equation*}%
\begin{equation*}
\left( \gamma _{1,2}^{2}\left( \varphi \left( \lambda \right) ^{3}-\varphi
\left( \lambda \right) ^{2}\right) +\gamma _{\left( 1,1\right) 2}\left(
\varphi \left( \lambda \right) ^{3}-\varphi \left( \lambda \right) \right)
\right) \left( x^{2}\right) x.
\end{equation*}%
Therefore we have the decomposition%
\begin{equation*}
\sigma \left( id_{k},\gamma _{1,2},\gamma _{1\left( 2,2\right) },\gamma
_{\left( 1,1\right) 2},1,0\right) \sigma \left( \varphi ,0,0,0,\alpha
_{1,2},\alpha _{2,1}\right) =
\end{equation*}%
\begin{equation}
\sigma \left( \varphi ,\gamma _{1,2},\gamma _{1\left( 2,2\right) },\gamma
_{\left( 1,1\right) 2},\alpha _{1,2},\alpha _{2,1}\right) .
\label{homom_decomp_1_4}
\end{equation}%
We can prove as in the Subsection \ref{nilp3} that the homorphisms of the
kind $\sigma \left( \varphi ,0,0,0,\alpha _{1,2},\alpha _{2,1}\right) $ are
invertible. Now we will research the homorphisms of the kind $\sigma \left(
id_{k},\gamma _{1,2},\gamma _{1\left( 2,2\right) },\gamma _{\left(
1,1\right) 2},1,0\right) $. We will consider two system of words $W^{\left(
1\right) }$ and $W^{\left( 2\right) }$ and two system of homomorphisms $%
\left\{ \sigma _{F}^{\left( 1\right) }:F\rightarrow F_{W^{\left( 1\right)
}}^{\ast }\mid F\in \mathrm{Ob}\Theta _{4}^{0}\right\} $ and $\left\{ \sigma
_{F}^{\left( 2\right) }:F\rightarrow F_{W^{\left( 2\right) }}^{\ast }\mid
F\in \mathrm{Ob}\Theta _{4}^{0}\right\} $ such that $\sigma _{F}^{\left(
i\right) }=\sigma \left( id_{k},\gamma _{1,2}^{\left( i\right) },\gamma
_{1\left( 2,2\right) }^{\left( i\right) },\gamma _{\left( 1,1\right)
2}^{\left( i\right) },1,0\right) $, $i=1,2$. We have that%
\begin{equation*}
\sigma \left( id_{k},\gamma _{1,2}^{\left( 1\right) },\gamma _{1\left(
2,2\right) }^{\left( 1\right) },\gamma _{\left( 1,1\right) 2}^{\left(
1\right) },1,0\right) \sigma \left( id_{k},\gamma _{1,2}^{\left( 2\right)
},\gamma _{1\left( 2,2\right) }^{\left( 2\right) },\gamma _{\left(
1,1\right) 2}^{\left( 2\right) },1,0\right) \left( x_{1}+x_{2}\right) =
\end{equation*}%
\begin{equation*}
\left( id_{k},\gamma _{1,2}^{\left( 2\right) },\gamma _{1\left( 2,2\right)
}^{\left( 2\right) },\gamma _{\left( 1,1\right) 2}^{\left( 2\right)
},1,0\right) (x_{1}+x_{2}+\gamma _{1,2}^{\left( 1\right) }x_{1}x_{2}+\gamma
_{1,2}^{\left( 1\right) }x_{2}x_{1}+
\end{equation*}%
\begin{equation*}
\left( \left( \gamma _{1,2}^{\left( 1\right) }\right) ^{2}+\gamma _{1\left(
2,2\right) }^{\left( 1\right) }\right) x_{1}\left( x_{1}x_{2}\right) +\left(
\left( \gamma _{1,2}^{\left( 1\right) }\right) ^{2}+\gamma _{1\left(
2,2\right) }^{\left( 1\right) }\right) x_{1}\left( x_{2}x_{1}\right) +
\end{equation*}%
\begin{equation*}
\gamma _{1\left( 2,2\right) }^{\left( 1\right) }x_{1}\left( x_{2}^{2}\right)
+\gamma _{1\left( 2,2\right) }^{\left( 1\right) }x_{2}\left(
x_{1}^{2}\right) +
\end{equation*}%
\begin{equation*}
\left( \left( \gamma _{1,2}^{\left( 1\right) }\right) ^{2}+\gamma _{1\left(
2,2\right) }^{\left( 1\right) }\right) x_{2}\left( x_{1}x_{2}\right) +\left(
\left( \gamma _{1,2}^{\left( 1\right) }\right) ^{2}+\gamma _{1\left(
2,2\right) }^{\left( 1\right) }\right) x_{2}\left( x_{2}x_{1}\right) +
\end{equation*}%
\begin{equation*}
\left( \left( \gamma _{1,2}^{\left( 1\right) }\right) ^{2}+\gamma _{\left(
1,1\right) 2}^{\left( 1\right) }\right) \left( x_{1}x_{2}\right)
x_{1}+\left( \left( \gamma _{1,2}^{\left( 1\right) }\right) ^{2}+\gamma
_{\left( 1,1\right) 2}^{\left( 1\right) }\right) \left( x_{2}x_{1}\right)
x_{1}+
\end{equation*}%
\begin{equation*}
\gamma _{\left( 1,1\right) 2}^{\left( 1\right) }\left( x_{2}^{2}\right)
x_{1}+\gamma _{\left( 1,1\right) 2}^{\left( 1\right) }\left(
x_{1}^{2}\right) x_{2}+
\end{equation*}%
\begin{equation*}
\left( \left( \gamma _{1,2}^{\left( 1\right) }\right) ^{2}+\gamma _{\left(
1,1\right) 2}^{\left( 1\right) }\right) \left( x_{1}x_{2}\right)
x_{2}+\left( \left( \gamma _{1,2}^{\left( 1\right) }\right) ^{2}+\gamma
_{\left( 1,1\right) 2}^{\left( 1\right) }\right) \left( x_{2}x_{1}\right)
x_{2}).
\end{equation*}%
\begin{equation*}
\sigma \left( id_{k},\gamma _{1,2}^{\left( 1\right) },\gamma _{1\left(
2,2\right) }^{\left( 1\right) },\gamma _{\left( 1,1\right) 2}^{\left(
1\right) },1,0\right) \left( \gamma _{1,2}^{\left( 1\right)
}x_{1}x_{2}\right) =
\end{equation*}%
\begin{equation*}
\gamma _{1,2}^{\left( 1\right) }\left( x_{1}x_{2}-\gamma _{1,2}^{\left(
2\right) }x_{1}\left( x_{2}^{2}\right) -\gamma _{1,2}^{\left( 2\right)
}\left( x_{1}^{2}\right) x_{2}\right) .
\end{equation*}%
\begin{equation*}
\sigma \left( id_{k},\gamma _{1,2}^{\left( 1\right) },\gamma _{1\left(
2,2\right) }^{\left( 1\right) },\gamma _{\left( 1,1\right) 2}^{\left(
1\right) },1,0\right) \left( \left( \left( \gamma _{1,2}^{\left( 1\right)
}\right) ^{2}+\gamma _{1\left( 2,2\right) }^{\left( 1\right) }\right)
x_{1}\left( x_{1}x_{2}\right) \right) =
\end{equation*}%
\begin{equation*}
\left( \left( \gamma _{1,2}^{\left( 1\right) }\right) ^{2}+\gamma _{1\left(
2,2\right) }^{\left( 1\right) }\right) x_{1}\left( x_{1}x_{2}\right) .
\end{equation*}%
So%
\begin{equation*}
\sigma \left( id_{k},\gamma _{1,2}^{\left( 1\right) },\gamma _{1\left(
2,2\right) }^{\left( 1\right) },\gamma _{\left( 1,1\right) 2}^{\left(
1\right) },1,0\right) \sigma \left( id_{k},\gamma _{1,2}^{\left( 2\right)
},\gamma _{1\left( 2,2\right) }^{\left( 2\right) },\gamma _{\left(
1,1\right) 2}^{\left( 2\right) },1,0\right) \left( x_{1}+x_{2}\right) =
\end{equation*}%
\begin{equation*}
(x_{1}+x_{2}+\gamma _{1,2}^{\left( 2\right) }x_{1}x_{2}+\gamma
_{1,2}^{\left( 2\right) }x_{2}x_{1}+
\end{equation*}%
\begin{equation*}
\left( \left( \gamma _{1,2}^{\left( 2\right) }\right) +\gamma _{1\left(
2,2\right) }^{\left( 2\right) }\right) x_{1}\left( x_{1}x_{2}\right) +\left(
\left( \gamma _{1,2}^{\left( 2\right) }\right) ^{2}+\gamma _{1\left(
2,2\right) }^{\left( 2\right) }\right) x_{1}\left( x_{2}x_{1}\right) +
\end{equation*}%
\begin{equation*}
\gamma _{1\left( 2,2\right) }^{\left( 2\right) }x_{1}\left( x_{2}^{2}\right)
+\gamma _{1\left( 2,2\right) }^{\left( 2\right) }x_{2}\left(
x_{1}^{2}\right) +
\end{equation*}%
\begin{equation*}
\left( \left( \gamma _{1,2}^{\left( 2\right) }\right) ^{2}+\gamma _{1\left(
2,2\right) }^{\left( 2\right) }\right) x_{2}\left( x_{1}x_{2}\right) +\left(
\left( \gamma _{1,2}^{\left( 2\right) }\right) ^{2}+\gamma _{1\left(
2,2\right) }^{\left( 2\right) }\right) x_{2}\left( x_{2}x_{1}\right) +
\end{equation*}%
\begin{equation*}
\left( \left( \gamma _{1,2}^{\left( 2\right) }\right) ^{2}+\gamma _{\left(
1,1\right) 2}^{\left( 2\right) }\right) \left( x_{1}x_{2}\right)
x_{1}+\left( \left( \gamma _{1,2}^{\left( 2\right) }\right) ^{2}+\gamma
_{\left( 1,1\right) 2}^{\left( 2\right) }\right) \left( x_{2}x_{1}\right)
x_{1}+
\end{equation*}%
\begin{equation*}
\gamma _{\left( 1,1\right) 2}^{\left( 2\right) }\left( x_{2}^{2}\right)
x_{1}+\gamma _{\left( 1,1\right) 2}^{\left( 2\right) }\left(
x_{1}^{2}\right) x_{2}+
\end{equation*}%
\begin{equation*}
\left( \left( \gamma _{1,2}^{\left( 2\right) }\right) ^{2}+\gamma _{\left(
1,1\right) 2}^{\left( 2\right) }\right) \left( x_{1}x_{2}\right)
x_{2}+\left( \left( \gamma _{1,2}^{\left( 2\right) }\right) ^{2}+\gamma
_{\left( 1,1\right) 2}^{\left( 2\right) }\right) \left( x_{2}x_{1}\right)
x_{2})\underset{\left( 2\right) }{\perp }
\end{equation*}%
\begin{equation*}
\left( \gamma _{1,2}^{\left( 1\right) }\left( x_{1}x_{2}-\gamma
_{1,2}^{\left( 2\right) }x_{1}\left( x_{2}^{2}\right) -\gamma _{1,2}^{\left(
2\right) }\left( x_{1}^{2}\right) x_{2}\right) +\gamma _{1,2}^{\left(
1\right) }\left( x_{2}x_{1}-\gamma _{1,2}^{\left( 2\right) }x_{2}\left(
x_{1}^{2}\right) -\gamma _{1,2}^{\left( 2\right) }\left( x_{2}^{2}\right)
x_{1}\right) \right) +
\end{equation*}%
\begin{equation*}
\left( \left( \gamma _{1,2}^{\left( 1\right) }\right) ^{2}+\gamma _{1\left(
2,2\right) }^{\left( 1\right) }\right) x_{1}\left( x_{1}x_{2}\right) +\left(
\left( \gamma _{1,2}^{\left( 1\right) }\right) ^{2}+\gamma _{1\left(
2,2\right) }^{\left( 1\right) }\right) x_{1}\left( x_{2}x_{1}\right) +
\end{equation*}%
\begin{equation*}
\gamma _{1\left( 2,2\right) }^{\left( 1\right) }x_{1}x_{2}^{2}+\gamma
_{1\left( 2,2\right) }^{\left( 1\right) }x_{2}x_{1}^{2}+
\end{equation*}%
\begin{equation*}
\left( \left( \gamma _{1,2}^{\left( 1\right) }\right) ^{2}+\gamma _{1\left(
2,2\right) }^{\left( 1\right) }\right) x_{2}\left( x_{1}x_{2}\right) +\left(
\left( \gamma _{1,2}^{\left( 1\right) }\right) ^{2}+\gamma _{1\left(
2,2\right) }^{\left( 1\right) }\right) x_{2}\left( x_{2}x_{1}\right) +
\end{equation*}%
\begin{equation*}
\left( \left( \gamma _{1,2}^{\left( 1\right) }\right) ^{2}+\gamma _{\left(
1,1\right) 2}^{\left( 1\right) }\right) \left( x_{1}x_{2}\right)
x_{1}+\left( \left( \gamma _{1,2}^{\left( 1\right) }\right) ^{2}+\gamma
_{\left( 1,1\right) 2}^{\left( 1\right) }\right) \left( x_{2}x_{1}\right)
x_{1}+
\end{equation*}%
\begin{equation*}
\gamma _{\left( 1,1\right) 2}^{\left( 1\right) }x_{2}^{2}x_{1}+\gamma
_{\left( 1,1\right) 2}^{\left( 1\right) }x_{1}^{2}x_{2}+
\end{equation*}%
\begin{equation*}
\left( \left( \gamma _{1,2}^{\left( 1\right) }\right) ^{2}+\gamma _{\left(
1,1\right) 2}^{\left( 1\right) }\right) \left( x_{1}x_{2}\right)
x_{2}+\left( \left( \gamma _{1,2}^{\left( 1\right) }\right) ^{2}+\gamma
_{\left( 1,1\right) 2}^{\left( 1\right) }\right) \left( x_{2}x_{1}\right)
x_{2}=
\end{equation*}%
\begin{equation*}
x_{1}+x_{2}+\gamma _{1,2}^{\left( 2\right) }x_{1}x_{2}+\gamma _{1,2}^{\left(
2\right) }x_{2}x_{1}+
\end{equation*}%
\begin{equation*}
\left( \left( \gamma _{1,2}^{\left( 2\right) }\right) +\gamma _{1\left(
2,2\right) }^{\left( 2\right) }\right) x_{1}\left( x_{1}x_{2}\right) +\left(
\left( \gamma _{1,2}^{\left( 2\right) }\right) ^{2}+\gamma _{1\left(
2,2\right) }^{\left( 2\right) }\right) x_{1}\left( x_{2}x_{1}\right) +
\end{equation*}%
\begin{equation*}
\gamma _{1\left( 2,2\right) }^{\left( 2\right) }x_{1}\left( x_{2}^{2}\right)
+\gamma _{1\left( 2,2\right) }^{\left( 2\right) }x_{2}\left(
x_{1}^{2}\right) +
\end{equation*}%
\begin{equation*}
\left( \left( \gamma _{1,2}^{\left( 2\right) }\right) ^{2}+\gamma _{1\left(
2,2\right) }^{\left( 2\right) }\right) x_{2}\left( x_{1}x_{2}\right) +\left(
\left( \gamma _{1,2}^{\left( 2\right) }\right) ^{2}+\gamma _{1\left(
2,2\right) }^{\left( 2\right) }\right) x_{2}\left( x_{2}x_{1}\right) +
\end{equation*}%
\begin{equation*}
\left( \left( \gamma _{1,2}^{\left( 2\right) }\right) ^{2}+\gamma _{\left(
1,1\right) 2}^{\left( 2\right) }\right) \left( x_{1}x_{2}\right)
x_{1}+\left( \left( \gamma _{1,2}^{\left( 2\right) }\right) ^{2}+\gamma
_{\left( 1,1\right) 2}^{\left( 2\right) }\right) \left( x_{2}x_{1}\right)
x_{1}+
\end{equation*}%
\begin{equation*}
\gamma _{\left( 1,1\right) 2}^{\left( 2\right) }\left( x_{2}^{2}\right)
x_{1}+\gamma _{\left( 1,1\right) 2}^{\left( 2\right) }\left(
x_{1}^{2}\right) x_{2}+
\end{equation*}%
\begin{equation*}
\left( \left( \gamma _{1,2}^{\left( 2\right) }\right) ^{2}+\gamma _{\left(
1,1\right) 2}^{\left( 2\right) }\right) \left( x_{1}x_{2}\right)
x_{2}+\left( \left( \gamma _{1,2}^{\left( 2\right) }\right) ^{2}+\gamma
_{\left( 1,1\right) 2}^{\left( 2\right) }\right) \left( x_{2}x_{1}\right)
x_{2}+
\end{equation*}%
\begin{equation*}
\gamma _{1,2}^{\left( 1\right) }x_{1}x_{2}-\gamma _{1,2}^{\left( 1\right)
}\gamma _{1,2}^{\left( 2\right) }x_{1}\left( x_{2}^{2}\right) -\gamma
_{1,2}^{\left( 1\right) }\gamma _{1,2}^{\left( 2\right) }\left(
x_{1}^{2}\right) x_{2}+
\end{equation*}%
\begin{equation*}
\gamma _{1,2}^{\left( 1\right) }x_{2}x_{1}-\gamma _{1,2}^{\left( 1\right)
}\gamma _{1,2}^{\left( 2\right) }x_{2}\left( x_{1}^{2}\right) -\gamma
_{1,2}^{\left( 1\right) }\gamma _{1,2}^{\left( 2\right) }\left(
x_{2}^{2}\right) x_{1}+
\end{equation*}%
\begin{equation*}
\gamma _{1,2}^{\left( 2\right) }\gamma _{1,2}^{\left( 1\right) }\left(
x_{1}+x_{2}\right) \left( x_{1}x_{2}+x_{2}x_{1}\right) +\gamma
_{1,2}^{\left( 2\right) }\gamma _{1,2}^{\left( 1\right) }\left(
x_{1}x_{2}+x_{2}x_{1}\right) \left( x_{1}+x_{2}\right) +
\end{equation*}%
\begin{equation*}
\left( \left( \gamma _{1,2}^{\left( 1\right) }\right) ^{2}+\gamma _{1\left(
2,2\right) }^{\left( 1\right) }\right) x_{1}\left( x_{1}x_{2}\right) +\left(
\left( \gamma _{1,2}^{\left( 1\right) }\right) ^{2}+\gamma _{1\left(
2,2\right) }^{\left( 1\right) }\right) x_{1}\left( x_{2}x_{1}\right) +
\end{equation*}%
\begin{equation*}
\gamma _{1\left( 2,2\right) }^{\left( 1\right) }x_{1}x_{2}^{2}+\gamma
_{1\left( 2,2\right) }^{\left( 1\right) }x_{2}x_{1}^{2}+
\end{equation*}%
\begin{equation*}
\left( \left( \gamma _{1,2}^{\left( 1\right) }\right) ^{2}+\gamma _{1\left(
2,2\right) }^{\left( 1\right) }\right) x_{2}\left( x_{1}x_{2}\right) +\left(
\left( \gamma _{1,2}^{\left( 1\right) }\right) ^{2}+\gamma _{1\left(
2,2\right) }^{\left( 1\right) }\right) x_{2}\left( x_{2}x_{1}\right) +
\end{equation*}%
\begin{equation*}
\left( \left( \gamma _{1,2}^{\left( 1\right) }\right) ^{2}+\gamma _{\left(
1,1\right) 2}^{\left( 1\right) }\right) \left( x_{1}x_{2}\right)
x_{1}+\left( \left( \gamma _{1,2}^{\left( 1\right) }\right) ^{2}+\gamma
_{\left( 1,1\right) 2}^{\left( 1\right) }\right) \left( x_{2}x_{1}\right)
x_{1}+
\end{equation*}%
\begin{equation*}
\gamma _{\left( 1,1\right) 2}^{\left( 1\right) }x_{2}^{2}x_{1}+\gamma
_{\left( 1,1\right) 2}^{\left( 1\right) }x_{1}^{2}x_{2}+
\end{equation*}%
\begin{equation*}
\left( \left( \gamma _{1,2}^{\left( 1\right) }\right) ^{2}+\gamma _{\left(
1,1\right) 2}^{\left( 1\right) }\right) \left( x_{1}x_{2}\right)
x_{2}+\left( \left( \gamma _{1,2}^{\left( 1\right) }\right) ^{2}+\gamma
_{\left( 1,1\right) 2}^{\left( 1\right) }\right) \left( x_{2}x_{1}\right)
x_{2}=
\end{equation*}%
\begin{equation*}
x_{1}+x_{2}+\left( \gamma _{1,2}^{\left( 2\right) }+\gamma _{1,2}^{\left(
1\right) }\right) x_{1}x_{2}+\left( \gamma _{1,2}^{\left( 2\right) }+\gamma
_{1,2}^{\left( 1\right) }\right) x_{2}x_{1}+
\end{equation*}%
\begin{equation*}
\left( \left( \gamma _{1,2}^{\left( 2\right) }+\gamma _{1,2}^{\left(
1\right) }\right) ^{2}+\gamma _{1\left( 2,2\right) }^{\left( 2\right)
}-\gamma _{1,2}^{\left( 2\right) }\gamma _{1,2}^{\left( 1\right) }+\gamma
_{1\left( 2,2\right) }^{\left( 1\right) }\right) x_{1}\left(
x_{1}x_{2}\right) +
\end{equation*}%
\begin{equation*}
\left( \left( \gamma _{1,2}^{\left( 2\right) }+\gamma _{1,2}^{\left(
1\right) }\right) ^{2}+\gamma _{1\left( 2,2\right) }^{\left( 2\right)
}-\gamma _{1,2}^{\left( 2\right) }\gamma _{1,2}^{\left( 1\right) }+\gamma
_{1\left( 2,2\right) }^{\left( 1\right) }\right) x_{1}\left(
x_{2}x_{1}\right) +
\end{equation*}%
\begin{equation*}
\left( \gamma _{1\left( 2,2\right) }^{\left( 2\right) }-\gamma
_{1,2}^{\left( 1\right) }\gamma _{1,2}^{\left( 2\right) }+\gamma _{1\left(
2,2\right) }^{\left( 1\right) }\right) x_{1}\left( x_{2}^{2}\right) +\left(
\gamma _{1\left( 2,2\right) }^{\left( 2\right) }-\gamma _{1,2}^{\left(
1\right) }\gamma _{1,2}^{\left( 2\right) }+\gamma _{1\left( 2,2\right)
}^{\left( 1\right) }\right) x_{2}\left( x_{1}^{2}\right) +
\end{equation*}%
\begin{equation*}
\left( \left( \gamma _{1,2}^{\left( 2\right) }+\gamma _{1,2}^{\left(
1\right) }\right) ^{2}+\gamma _{1\left( 2,2\right) }^{\left( 2\right)
}-\gamma _{1,2}^{\left( 2\right) }\gamma _{1,2}^{\left( 1\right) }+\gamma
_{1\left( 2,2\right) }^{\left( 1\right) }\right) x_{2}\left(
x_{1}x_{2}\right) +
\end{equation*}%
\begin{equation*}
\left( \left( \gamma _{1,2}^{\left( 2\right) }+\gamma _{1,2}^{\left(
1\right) }\right) ^{2}+\gamma _{1\left( 2,2\right) }^{\left( 2\right)
}-\gamma _{1,2}^{\left( 2\right) }\gamma _{1,2}^{\left( 1\right) }+\gamma
_{1\left( 2,2\right) }^{\left( 1\right) }\right) x_{2}\left(
x_{2}x_{1}\right) +
\end{equation*}%
\begin{equation*}
\left( \left( \gamma _{1,2}^{\left( 2\right) }+\gamma _{1,2}^{\left(
1\right) }\right) ^{2}+\gamma _{\left( 1,1\right) 2}^{\left( 2\right)
}-\gamma _{1,2}^{\left( 2\right) }\gamma _{1,2}^{\left( 1\right) }+\gamma
_{\left( 1,1\right) 2}^{\left( 1\right) }\right) \left( x_{1}x_{2}\right)
x_{1}+
\end{equation*}%
\begin{equation*}
\left( \left( \gamma _{1,2}^{\left( 2\right) }+\gamma _{1,2}^{\left(
1\right) }\right) ^{2}+\gamma _{\left( 1,1\right) 2}^{\left( 2\right)
}-\gamma _{1,2}^{\left( 2\right) }\gamma _{1,2}^{\left( 1\right) }+\gamma
_{\left( 1,1\right) 2}^{\left( 1\right) }\right) \left( x_{2}x_{1}\right)
x_{1}+
\end{equation*}%
\begin{equation*}
\left( \gamma _{\left( 1,1\right) 2}^{\left( 2\right) }-\gamma
_{1,2}^{\left( 1\right) }\gamma _{1,2}^{\left( 2\right) }+\gamma _{\left(
1,1\right) 2}^{\left( 1\right) }\right) \left( x_{2}^{2}\right) x_{1}+\left(
\gamma _{\left( 1,1\right) 2}^{\left( 2\right) }-\gamma _{1,2}^{\left(
1\right) }\gamma _{1,2}^{\left( 2\right) }+\gamma _{\left( 1,1\right)
2}^{\left( 1\right) }\right) \left( x_{1}^{2}\right) x_{2}+
\end{equation*}%
\begin{equation*}
\left( \left( \gamma _{1,2}^{\left( 2\right) }+\gamma _{1,2}^{\left(
1\right) }\right) ^{2}+\gamma _{\left( 1,1\right) 2}^{\left( 2\right)
}-\gamma _{1,2}^{\left( 2\right) }\gamma _{1,2}^{\left( 1\right) }+\gamma
_{\left( 1,1\right) 2}^{\left( 1\right) }\right) \left( x_{1}x_{2}\right)
x_{2}+
\end{equation*}%
\begin{equation*}
\left( \left( \gamma _{1,2}^{\left( 2\right) }+\gamma _{1,2}^{\left(
1\right) }\right) ^{2}+\gamma _{\left( 1,1\right) 2}^{\left( 2\right)
}-\gamma _{1,2}^{\left( 2\right) }\gamma _{1,2}^{\left( 1\right) }+\gamma
_{\left( 1,1\right) 2}^{\left( 1\right) }\right) \left( x_{2}x_{1}\right)
x_{2}.
\end{equation*}%
\begin{equation*}
\sigma \left( id_{k},\gamma _{1,2}^{\left( 2\right) },\gamma _{1\left(
2,2\right) }^{\left( 2\right) },\gamma _{\left( 1,1\right) 2}^{\left(
2\right) },1,0\right) \sigma \left( id_{k},\gamma _{1,2}^{\left( 1\right)
},\gamma _{1\left( 2,2\right) }^{\left( 1\right) },\gamma _{\left(
1,1\right) 2}^{\left( 1\right) },1,0\right) \left( x_{1}x_{2}\right) \equiv 
\end{equation*}%
\begin{equation*}
\sigma \left( id_{k},\gamma _{1,2}^{\left( 2\right) },\gamma _{1\left(
2,2\right) }^{\left( 2\right) },\gamma _{\left( 1,1\right) 2}^{\left(
2\right) },1,0\right) \left( x_{1}x_{2}\right) \equiv x_{1}x_{2}\left( \func{%
mod}F^{3}\right) ,
\end{equation*}%
where $F=F^{\left( 4\right) }\left( x_{1},x_{2}\right) $. Therefore%
\begin{equation*}
\sigma \left( id_{k},\gamma _{1,2}^{\left( 2\right) },\gamma _{1\left(
2,2\right) }^{\left( 2\right) },\gamma _{\left( 1,1\right) 2}^{\left(
2\right) },1,0\right) \sigma \left( id_{k},\gamma _{1,2}^{\left( 1\right)
},\gamma _{1\left( 2,2\right) }^{\left( 1\right) },\gamma _{\left(
1,1\right) 2}^{\left( 1\right) },1,0\right) =
\end{equation*}%
\begin{equation*}
\sigma \left( id_{k},\gamma _{1,2}^{\left( 1\right) }+\gamma _{1,2}^{\left(
2\right) },\gamma _{1\left( 2,2\right) }^{\left( 1\right) }-\gamma
_{1,2}^{\left( 1\right) }\gamma _{1,2}^{\left( 2\right) }+\gamma _{1\left(
2,2\right) }^{\left( 2\right) },\gamma _{\left( 1,1\right) 2}^{\left(
1\right) }-\gamma _{1,2}^{\left( 1\right) }\gamma _{1,2}^{\left( 2\right)
}+\gamma _{\left( 1,1\right) 2}^{\left( 2\right) },1,0\right) .
\end{equation*}%
From this formula we can conclude that%
\begin{equation*}
\sigma \left( id_{k},\gamma _{1,2},\gamma _{1\left( 2,2\right) },\gamma
_{\left( 1,1\right) 2},1,0\right) =
\end{equation*}%
\begin{equation*}
\sigma \left( id_{k},\gamma _{1,2},0,0,1,0\right) \sigma \left(
id_{k},0,\gamma _{1\left( 2,2\right) },0,1,0\right) \sigma \left(
id_{k},0,0,\gamma _{\left( 1,1\right) 2},1,0\right) ,
\end{equation*}%
\begin{equation*}
\sigma \left( id_{k},0,0,\gamma _{\left( 1,1\right) 2},1,0\right)
^{-1}=\sigma \left( id_{k},0,0,-\gamma _{\left( 1,1\right) 2},1,0\right) ,
\end{equation*}%
\begin{equation*}
\sigma \left( id_{k},0,\gamma _{1\left( 2,2\right) },0,1,0\right)
^{-1}=\sigma \left( id_{k},0,-\gamma _{1\left( 2,2\right) },0,1,0\right) ,
\end{equation*}%
\begin{equation*}
\sigma \left( id_{k},\gamma _{1,2},0,0,1,0\right) ^{-1}=\sigma \left(
id_{k},-\gamma _{1,2},-\gamma _{1,2}^{2},-\gamma _{1,2}^{2},1,0\right) .
\end{equation*}%
These formulas and (\ref{homom_decomp_1_4}) allow to conclude that the all
systems of words $W$ defined by formulas (\ref{real_list_nilp3}), (\ref%
{add_4}), (\ref{mul_scal_4}) and (\ref{mul_4}) fulfill conditions Op1) and
Op2). So we calculated the group $\mathfrak{S}_{4}$.

Now we will calculate the group $\mathfrak{Y}_{4}\cap \mathfrak{S}_{4}$. We
will consider the strongly stable automorphism $\Phi $ defined by the system
of words $W$ which fulfills conditions Op1) and Op2). We will find when this
automorphism is inner. We can prove as in the Subsection \ref{nilp3} that if 
$\Phi $ is inner then $\alpha _{2,1}=0$, $\varphi =id_{k}$. Now we will
prove that this conditions are sufficient. For this purpose for every $F\in 
\mathrm{Ob}\Theta _{4}^{0}$ we define the mapping $c_{F}:F\rightarrow F$ by
this formula:%
\begin{equation*}
c_{F}\left( f\right) =\alpha _{1,2}^{-1}f+\alpha _{1,2}^{-2}\gamma
_{1,2}f^{2}+\alpha _{1,2}^{-3}\left( \gamma _{1,2}^{2}+\gamma _{1\left(
2,2\right) }\right) f\left( f^{2}\right) +\alpha _{1,2}^{-3}\left( \gamma
_{1,2}^{2}+\gamma _{\left( 1,1\right) 2}\right) \left( f^{2}\right) f,
\end{equation*}%
where $f\in F$. We can prove by direct calculation that $c_{F}:F\rightarrow
F_{W}^{\ast }$ is an isomorphism. For example, we can check that $%
c_{F}\left( f_{1}f_{2}\right) =c_{F}\left( f_{1}\right) \times c_{F}\left(
f_{2}\right) $, where $f_{1},f_{2}\in F$. We have by nilpotence that 
\begin{equation*}
c_{F}\left( f_{1}f_{2}\right) =\alpha _{1,2}^{-1}f_{1}f_{2}.
\end{equation*}%
\begin{equation*}
c_{F}\left( f_{1}\right) \times c_{F}\left( f_{2}\right) =
\end{equation*}%
\begin{equation*}
\left( \alpha _{1,2}^{-1}f_{1}+\alpha _{1,2}^{-2}\gamma
_{1,2}f_{1}^{2}+\alpha _{1,2}^{-3}\left( \gamma _{1,2}^{2}+\gamma _{1\left(
2,2\right) }\right) f_{1}\left( f_{1}^{2}\right) +\alpha _{1,2}^{-3}\left(
\gamma _{1,2}^{2}+\gamma _{\left( 1,1\right) 2}\right) \left(
f_{1}^{2}\right) f_{1}\right) \times 
\end{equation*}%
\begin{equation*}
\left( \alpha _{1,2}^{-1}f_{2}+\alpha _{1,2}^{-2}\gamma
_{1,2}f_{2}^{2}+\alpha _{1,2}^{-3}\left( \gamma _{1,2}^{2}+\gamma _{1\left(
2,2\right) }\right) f_{2}\left( f_{2}^{2}\right) +\alpha _{1,2}^{-3}\left(
\gamma _{1,2}^{2}+\gamma _{\left( 1,1\right) 2}\right) \left(
f_{2}^{2}\right) f_{2}\right) =
\end{equation*}%
\begin{equation*}
\alpha _{1,2}\left( \alpha _{1,2}^{-1}f_{1}+\alpha _{1,2}^{-2}\gamma
_{1,2}f_{1}^{2}\right) \left( \alpha _{1,2}^{-1}f_{2}+\alpha
_{1,2}^{-2}\gamma _{1,2}f_{2}^{2}\right) -
\end{equation*}%
\begin{equation*}
\alpha _{1,2}^{-2}\gamma _{1,2}f_{1}^{2}f_{2}-\alpha _{1,2}^{-2}\gamma
_{1,2}f_{1}f_{2}^{2}=
\end{equation*}%
\begin{equation*}
\alpha _{1,2}^{-1}f_{1}f_{2}+\alpha _{1,2}^{-2}\gamma
_{1,2}f_{1}f_{2}^{2}+\alpha _{1,2}^{-2}\gamma _{1,2}f_{1}^{2}f_{2}-
\end{equation*}%
\begin{equation*}
\alpha _{1,2}^{-2}\gamma _{1,2}f_{1}f_{2}^{2}-\alpha _{1,2}^{-2}\gamma
_{1,2}f_{1}^{2}f_{2}=
\end{equation*}%
\begin{equation*}
\alpha _{1,2}^{-1}f_{1}f_{2}.
\end{equation*}%
Also by direct calculations we can prove that $c_{F}\left(
f_{1}+f_{2}\right) =c_{F}\left( f_{1}\right) \perp c_{F}\left( f_{2}\right) $%
, where $f_{1},f_{2}\in F$, and $c_{F}\left( \lambda f\right) =\lambda \ast f
$, where $f\in F$, $\lambda \in k$. We can prove as in the Subsection \ref%
{nilp3} that $c_{F}$ is an isomorphism. It is clear that $c_{F}\psi =\psi
c_{D}$ fulfills for every $F,D\in \mathrm{Ob}\Theta _{4}^{0}$ and every $%
\left( \psi :D\rightarrow F\right) \in \mathrm{Mor}\Theta _{4}^{0}$.
Therefore, as in the \cite{Tsurkov} we can prove that 
\begin{equation*}
\mathfrak{A/Y\cong }\left( U\left( k\mathbf{S}_{\mathbf{2}}\right) \mathfrak{%
/}U\left( k\left\{ e\right\} \right) \right) \mathfrak{\leftthreetimes }%
\mathrm{Aut}k\mathfrak{\cong }k^{\ast }\mathfrak{\leftthreetimes }\mathrm{Aut%
}k.
\end{equation*}
\end{proof}

\section{Summary.\label{table}}

\setcounter{equation}{0}

The results of this paper and of the \cite{Tsurkov} can be summarized in
this table:

\begin{tabular}{|c|c|c|}
\hline
& Variety of the all & $\mathfrak{A/Y}$ \\ \hline
1 & linear algebras & $k^{\ast }\mathfrak{\leftthreetimes }\mathrm{Aut}k$ \\ 
\hline
2 & commutative algebras & $\mathrm{Aut}k$ \\ \hline
3 & power associative algebras & $k^{\ast }\mathfrak{\leftthreetimes }%
\mathrm{Aut}k$ \\ \hline
4 & alternative algebras & $\mathbf{S}_{\mathbf{2}}\times \mathrm{Aut}k$ \\ 
\hline
5 & Jordan algebras & $\mathrm{Aut}k$ \\ \hline
6 & 
\begin{tabular}{c}
arbitrary subvariety of anticommutative \\ 
algebras defined by identities with coefficients from $%
\mathbb{Z}
$%
\end{tabular}
& $\mathrm{Aut}k$ \\ \hline
7 & 
\begin{tabular}{l}
nilpotent algebras with degree \\ 
of nilpotence no more than $3$%
\end{tabular}
& $k^{\ast }\mathfrak{\leftthreetimes }\mathrm{Aut}k$ \\ \hline
8 & 
\begin{tabular}{l}
nilpotent algebras with degree \\ 
of nilpotence no more than $4$%
\end{tabular}
& $k^{\ast }\mathfrak{\leftthreetimes }\mathrm{Aut}k$ \\ \hline
\end{tabular}

\section{Acknowledgements.}

I am thankful to Prof. I. P. Shestakov who was very heedful to my research.

I acknowledge the support by FAPESP - Funda\c{c}\~{a}o de Amparo \`{a}
Pesquisa do Estado de S\~{a}o Paulo (Foundation for Support Research of the
State S\~{a}o Paulo), project No. 2010/50948-2.

\end{document}